\documentclass[reqno,12pt]{amsart}
\usepackage{amsmath, euscript, amsthm, epsfig}%, showkeys}
\font\msbm=msbm10
\font\msbmsmall=msbm10 scaled 800
\font\msbmtiny=msbm10 scaled 700
\newfam\msbmfam
\textfont\msbmfam=\msbm
\scriptfont\msbmfam=\msbmsmall
\scriptscriptfont\msbmfam=\msbmtiny

\def\PI{P$_{\mbox{\rm\tiny I}}$}
\def\PII{P$_{\mbox{\rm\tiny II}}$}

\def\PV{P$_{\mbox{\rm\tiny V}}$}
\def\PVI{P$_{\mbox{\rm\tiny VI}}$}
\def\PItwo{P$_{\mbox{\rm\tiny I}}^2$}
\def\PItwotwo{P$_{\mbox{\rm\tiny I}}^{(2,2)}$}
\def\PItwone{P$_{\mbox{\rm\tiny I}}^{(2,1)}$}
\def\PItwozero{P$_{\mbox{\rm\tiny I}}^{(2,0)}$}

\def\LItwotwo{L$_{\mbox{\rm\tiny I}}^{(2,2)}$}
\def\LItwone{L$_{\mbox{\rm\tiny I}}^{(2,1)}$}
\def\LItwozero{L$_{\mbox{\rm\tiny I}}^{(2,0)}$}

\def\Re{\mbox{\rm Re}\,}
\def\Im{\mbox{\rm Im}\,}
\def\Ai{\mbox{\rm Ai}\,}
\def\sgn{\mbox{\,\rm sgn}}

\def\res\mathop{\mbox{\,\rm res}}

\def\varkappa{\mbox{\msbm\char'173}}      %\varkappa

\def\res{\mathop{\hbox{res}}}

\theoremstyle{plain}
\newtheorem{thm}{Theorem}[section]
\newtheorem{prop}[thm]{Proposition}

\newtheorem{RHP}{Riemann--Hilbert problem}

\theoremstyle{definition}

\theoremstyle{remark}
\newtheorem{rem}{Remark}[section]
\newtheorem{conj}{Conjecture}[section]

\numberwithin{equation}{section}

\numberwithin{figure}{section}

\begin{document}
\thispagestyle{empty}

\title[On a non-Painlev\'e isomonodromy deformation equation]
{On an isomonodromy deformation equation without 
the Painlev\'e property}
\author{Boris Dubrovin}
\address{SISSA, Via Bonomea 265, 34136, Trieste, Italy, and Laboratory of Geometric Methods in Mathematical
Physics, Moscow State University `M.V.Lomonosov'}
\email{dubrovin@sissa.it}
\author{Andrei Kapaev}
\address{SISSA, Via Bonomea 265, 34136, Trieste, Italy}
\email{akapaev@sissa.it}

\begin{abstract}
We show that the fourth order nonlinear ODE which controls 
the pole dynamics in the general solution of equation \PItwo\
compatible with the KdV equation exhibits two remarkable properties:
1)~it governs the isomonodromy deformations of a $2\times2$ matrix 
linear ODE with polynomial coefficients, and 2)~it does not possesses 
the Painlev\'e property. We also study the properties of the 
Riemann--Hilbert problem associated to this ODE and find its large $t$
asymptotic solution for the physically interesting initial data.
\end{abstract}
\maketitle

\section{Introduction}

The study of relationships between the theory of isomonodromic deformations and the theory of differential equations satisfying the so-called Painlev\'e property is a well-established branch of the analytic theory of differential equations in complex domain (see below a brief summary of the most important results obtained in this direction). One of the outputs of the present paper suggests that the above mentioned relationship is less straightforward than it was traditionally believed. We illustrate the point with an example of a fourth order ODE for a function $a=a(t)$
\begin{equation}\label{PI21}
a''''
+120(a')^3a''
-120a'a''t
-\tfrac{200}{3}(a')^2
-\tfrac{40}{3}aa''
+\tfrac{200}{9}t=0.
\tag{\PItwone}
\end{equation}
This equation appeared in \cite{shimomura} in the study of pole loci of solutions to a degenerate Garnier system.
Our observation is that eq. \PItwone\ governs isomonodromic deformation of certain linear differential operator with polynomial coefficients  (see eq. \eqref{main1} below). However, this equation does not satisfy the Painlev\'e property as its general solution has third order branch points of the form
\begin{equation}\label{ord3}
a(t) =a_0 - (t-b)^{1/3} +{\mathcal O}\left( (t-b)^{5/3}\right), \quad b\in\mathbb C
\end{equation}
(cf. \cite{shimomura}) where the location of the branch points depends on the choice of the solution (the so-called \emph{movable critical singularities}).

In spite of such a somewhat surprising phenomenon the method of isomonodromic deformations proves to be almost as powerful in the study of solutions to the equation \ref{PI21} as in the case of classical Painlev\'e equations. Namely, it is possible to derive the large $t$ asymptotics of solutions and, moreover, to describe the branching locus of a given solution in terms of a kind of a spectral problem for a quintic anharmonic oscillator.

The equation \ref{PI21} is of interest on its own. Namely, it describes the behavior of poles of solutions to another fourth order ODE
\begin{equation}\label{pp12}
u_{xxxx}
+10u_x^2
+20uu_{xx}
+40
\bigl(
u^3
-6tu
+6x
\bigr)
=0
\end{equation}
usually denoted as \PItwo. It is the second member of the so-called
\PI\ hierarchy. The coefficients of this equation depend on $t$ as on a parameter. It is well known (see, e.g., \cite{moore}) that the equation \eqref{pp12} is compatible with the Korteweg - de Vries (KdV) dynamics
\begin{equation}\label{KKdV}
u_t+uu_x+\tfrac{1}{12}u_{xxx}=0.
\end{equation}
The Laurent expansion of solutions $u=u(x,t)$ to the system \eqref{pp12}, \eqref{KKdV} near a pole $x=a(t)$ has the form
\begin{equation}\label{laurent}
u=-\frac1{(x-a(t))^2} +{\mathcal O}\left( (x-a(t))^2\right)
\end{equation}
where the function $a(t)$ solves\footnote{The connection of eq. \PItwone \ with the KdV equation was not considered in \cite{shimomura}.} eq. \ref{PI21}. In this way one associates a multivalued solution $a(t)$ to eq. \ref{PI21} with any solution to eq. \PItwo. The branch points \eqref{ord3} of $a(t)$ correspond to the triple collisions of the poles of solutions to the \PItwo. Thus, the above mentioned isomonodromy realization of eq. \ref{PI21} provides one with a tool for studying the KdV dynamics of poles of solutions to the \PItwo ~equation.

Particular solutions to the \PItwo ~equation are of interest. We will concentrate our attention on one of them, namely, the one that has no poles on the real axis $x\in \mathbb R$ (see \cite{BMP}, \cite{D1}, \cite{D2}, \cite{CG} about importance of this special solution to the eq. \PItwo). Such a solution $u(x,t)$ exists \cite{CV} for any real $t$ and it is uniquely determined by its asymptotic behavior for large $|x|$. Using the developed techniques along with the isomonodromic description \cite{K} of the special solution to \PItwo, we arrive at the asymptotic description for large $t$ of poles of this special solution. 
%We argue that there are no triple collisions of poles for any $t\in \mathbb C$ for this special solution to eq.\ \PItwo. Such a conjecture will be thoroughly investigated in subsequent publications.

\medskip

Before proceeding to the formulation of main results
let us first briefly recall the basic 
notions of the present paper, namely the Painlev\'e property and 
the isomonodromic deformations. In the description of the historical framework we will manly follow the paper \cite{Gray}.

\subsection{Painlev\'e property}
In 1866, L.~Fuchs \cite{LFuchs0} has shown that all the singular points 
of solutions to a linear ODE are among the singularities of its coefficients
and thus are independent from the initial conditions. In the nonlinear case,
the reasonable problem is to look for ODEs defining the families of functions, 
called the {\em general solutions}, which can be meromorphically extended to 
the universal covering space of a punctured Riemann surface with the punctures 
determined by the equation. In other words, the problem is to find ODEs whose
general solutions are free from the branch points and essential singularities
depending on the specific choice of the initial data. This property is called now 
the {\em Painlev\'e property}, or the {\em analytic Painlev\'e property}, and is 
obviously shared not only by the linear ODEs but also by the ODEs for 
elliptic functions.

In \cite{LFuchs}, L.~Fuchs started the classification of the first order polynomial 
in $u$ and $u'$ ODEs, $F(x,u,u')=0$, with single-valued in $x$ coefficients
with respect to the Painlev\'e property. H.~Poincar\'e \cite{Poincare} and 
P.~Painlev\'e \cite{Painleve} accomplished the analysis not finding new 
transcendental functions. 

In \cite{P}, P.~Painlev\'e revisited the L.~Fuchs' idea extending 
the program of classification to the second order ODEs of the form
$u''=F(x,u,u')$ with $F$ meromorphic in $x$ and rational in $u$ and 
$u'$. In the course of classification of the 2nd order 1st degree
ODEs modulo the M\"obius transformation \cite{Gam, I}, it was 
found 50 equations which pass the Painlev\'e $\alpha$-test and thus
called now the {\em Painlev\'e--Gambier equations}. It occurs that 
all these equations can be either integrated in terms of the classical 
linear transcendents, or elliptic functions, or reduce to one of the six 
exceptional {\em classical Painlev\'e equations}
\PI--\PVI. 

\subsection{Isomonodromic deformations}
The monodromy group to a linear ODE was first considered by B.~Riemann 
\cite{Rim}, H.~Schwarz \cite{Schwarz} and H.~Poincar\'e \cite{Poincare2}. 
Apparently, it was L.~Fuchs in \cite{LFuchs2} who for the first time set
the problem of the deformations of the coefficients in a linear equation that leave 
the monodromy group unchanged. Namely, assuming that solutions of 
a linear ODE depend on an additional variable, he obtained a system 
of the first order PDEs the solutions must satisfy.

More modern treatment of the isomonodromy deformations was developed by 
R.~Fuchs in \cite{F}. He has shown that the monodromy group of a scalar 
linear ODE with four Fuchsian singularities at $\lambda=0,1,\infty,x$ 
and an apparent singular point at $\lambda=u$ does not depend on the 
location $x$ of the fourth Fuchsian singular point if the location of 
the apparent singularity $u$ depends on $x$ according to a nonlinear 2nd 
order ODE. Later, the Fuchs' isomonodromic deformation equation was 
included as the sixth Painlev\'e equation \PVI\ into the list of the 
classical Painlev\'e equations.

In 1912, L.~Schlesinger \cite{Sch} generalized the Fuchs' approach
finding the equations of the isomonodromic deformations for arbitrary
linear Fuchsian ODEs, while R.~Garnier \cite{Gar} presented the scalar 
second order linear ODEs with irregular singular points whose 
isomonodromic deformations are governed by the lower classical 
Painlev\'e equations \PI--\PV. 
In 1980, M.~Jimbo, T.~Miwa and K.~Ueno 
\cite{JMU} extended the theory of Garnier to the linear ODEs with 
generic irregular singularities.

\subsection{Painlev\'e property of the isomonodromy deformation
equations}

It is interesting that to the date of the mentioned above achievements,
the fact that these equations do indeed possess the Painlev\'e property 
was not proved rigorously even for the case of the classical Painlev\'e
equations. Thus the term ``Painlev\'e equation'' loosely refers 
to various equations among which we mention the higher order ODEs in 
hierarchies associated with the classical Painlev\'e equations, 
the higher order ODEs in the classifications by F.~Bureau \cite{Bur} 
and C.~Cosgrove \cite{Cos} based on the use of some Painlev\'e tests, 
as well as the differential, difference, $q$-difference and 
elliptic-difference equations found in the course of the study of 
the symmetries and geometry of the classical Painlev\'e equations, 
see e.g.\ \cite{Noumi} and \cite{Sakai}.

A general elegant approach to the Painlev\'e property of the equations 
of the isomonodromic deformations was presented by T.~Miwa \cite{Miw} 
and B.~Malgrange \cite{Ma} in the early 80s of the last century
(recall that the direct proof of the Painlev\'e property of the classical
Painlev\'e equations appeared even later, see \cite{GLSh}).
The approach by T.~Miwa and B.~Malgrange is based on the use of the zero 
curvature representation and the Riemann--Hilbert correspondence. 
In fact, they have proved the analytic Painlev\'e property of 
the isomonodromic deformation equations for arbitrary linear Fuchsian 
ODEs and for equations with unbranched irregular singular points. 
In this respect we also mention the papers of M.~Inaba and M.~Saito 
\cite{IS} who developed an algebro-geometric approach to 
the {\em geometric Painlev\'e property} of the isomonodromic deformations
of the logarithmic and unramified irregular connections which also
implies their analytic Painlev\'e property.

\subsection{Plan of the paper and the main results}

One can expect that the isomonodromy deformation equations for arbitrary
rational connections possess the Painlev\'e property. However, the na\"ive 
induction does not work. Namely, and this is our first result, we show that
the polynomial in all  variables 4th order 1st degree ODE \eqref{PI21} 
%\begin{equation}\label{PI21}
%a''''
%+120(a')^3a''
%-120a'a''t
%-\tfrac{200}{3}(a')^2
%-\tfrac{40}{3}aa''
%+\tfrac{200}{9}t=0.
%\tag{\PItwone}
%\end{equation}
governs the isomonodromic deformations of the following linear differential equation with polynomial coefficients
\begin{equation}\label{main1}
\frac{d\,Z}{d\lambda}=\left(\begin{array}{cc} -\frac3{20} a'' & \frac1{30} \lambda^4+\alpha_3 \lambda^3 +\alpha_2 \lambda^2 +\alpha_1\lambda +\alpha_0 \\
\\
\frac1{30} \lambda-\alpha_3 & \frac3{20} a''\end{array}\right)\, Z
\end{equation}
%\begin{eqnarray}
%&&\label{cmain}
%{\mathcal C}=\left(\begin{array}{cc} 0 & -\frac13 \lambda^3 +\beta_2 \lambda^2 +\beta_1\lambda +\beta_0\\
%\\
%-\frac13 & 0\end{array}\right).
%\nonumber
%\end{eqnarray}
Here $a=a(t)$, $\alpha_k=\alpha_k(t)$, $k=0, \, 1, \, 2, \, 3$ are some smooth functions.

\begin{prop} The monodromy\footnote{As the system \eqref{main1} has only an irregular singularity at infinity, the isomonodromicity means independence from $t$ of the Stokes multipliers of the system - see below.} of the system \eqref{main1} does not depend on $t$ if and only if the coefficients have the form
\begin{eqnarray}
&&
\alpha_3=-\frac1{10} a'+\alpha_3^0, \quad \alpha_2=\frac3{10} {a'}^2 -6 \alpha_3^0 a' -t + \alpha_2^0
\nonumber\\
&&
\alpha_1=-\frac9{10} {a'}^3 +27 \alpha_3^0 {a'}^2 -3[\alpha_2^0 +60(\alpha_3^0)^2 -t] a' +a-40 \alpha_3^0 t +\alpha_1^0
\nonumber\\
&&
\alpha_0=-\frac9{20} a'''-\frac{54}5 {a'}^4 +432 \alpha_3^0 {a'}^3 + 18 \left[ t-330 (\alpha_3^0)^2  - \alpha_2^0 \right]{a'}^2
\nonumber\\
&&
+3\left[  a -130 \alpha_3^0 t +10800 (\alpha_3^0)^3 +90 \alpha_2^0 \alpha_3^0 +  \alpha_1^0\right]\, a' -30 \alpha_3^0\, a \nonumber\\
&&
+2100(\alpha_3^0)^2 t -54000 (\alpha_3^0)^4 -900\alpha_2^0(\alpha_3^0)^2 -30 \alpha_1^0 \alpha_3^0 
\nonumber
\end{eqnarray}
%\begin{eqnarray}
%&&
%\beta_2=2 a'-20 \alpha_3^0, \quad \beta_1=-9{a'}^2 +180 \alpha_3^0 a' +10 t -600(\alpha_3^0)^2 -10 \alpha_2^0
%&&
%\beta_0=36 {a'}^3 -1080 \alpha_3^0 {a'}^2 +60[\alpha_2^0+150 (\alpha_3^0)^2 -t]a'
%\nonumber\\
%&&
%\qquad -10 [a-70 \alpha_3^0 t +1800(\alpha_3^0)^3 +30 \alpha_2^0\alpha_3^0+\alpha_1^0]
%\nonumber
%\nonumber
%\end{eqnarray}
where $\alpha_1^0$, $\alpha_2^0$, $\alpha_3^0$ are some constants,
together with the following ODE for the function $a=a(t)$
\begin{eqnarray}\label{ode-general}
&&
a^{IV} +\left[120 {a'}^3 -3600 \alpha_3^0 {a'}^2 +120 \left( \alpha_2^0 +270 (\alpha_3^0)^2 -t\right) a' 
\right.
\nonumber\\
&&\left.
-\frac{40}3 \left( a-100 \alpha_3^0 t + 6300 (\alpha_3^0)^3 +60 \alpha_2^0 \alpha_3^0 +\alpha_1^0\right) \right] a'' -\frac{200}3 {a'}^2 +\frac{200}9 t
\nonumber\\
&&
+\frac{4000}3 \alpha_3^0 a' -6000 (\alpha_3^0)^2 -\frac{200}9 \alpha_2^0=0.
\end{eqnarray}
\end{prop}

After a change
$$
t\mapsto t+\alpha_2^0 -30 (\alpha_3^0)^2, \quad a\mapsto a +10 \alpha_3^0 t -\alpha_1^0 +40 \alpha_2^0 \alpha_3^0 -300(\alpha_3^0)^3
$$
the equation \eqref{ode-general} reduces to the 
 \ref{PI21} equation.
 
So, our claim is that the equation  \ref{PI21} describes the isomonodromic deformations of the system \eqref{main1} but, however, it does not possess the Painlev\'e property
as being satisfied by a 4-parameter Puiseux series in powers
of $(t-b)^{1/3}$. Nevertheless, as we show, this equation can be
effectively analyzed using the isomonodromy deformation techniques 
developed during last decades. In particular, the ramification points $t=b$ (see eq. \eqref{ord3} above) can be determined by a kind of a spectral problem for the quintic anharmonic oscillator
\begin{eqnarray}\label{oscil}
&&
\frac{d^2 y}{d\lambda^2} =\tfrac1{30} U(\lambda)\, y
\nonumber\\
&&
\\
&&
U(\lambda)=
\tfrac{1}{30}\lambda^5
-b\lambda^3
+a_0\lambda^2
+(\tfrac{360}{49}b^2+\tfrac{33}{10}a_9)\lambda
-\tfrac{230}{21}a_0b
+\tfrac{143}{30}a_{11}.
\nonumber
\end{eqnarray}
Here $a_0$, $a_9$, $a_{11}$ are the coefficients of the Puiseaux 
expansion of $a(t)$ (see eq. \eqref{a_Puiseux_series} below). 
Namely, the coefficients of the quintic polynomial must be chosen 
in such a way that the equation \eqref{oscil} possesses solutions 
exponentially decaying on certain contours in the complex plane, 
look for details in Remark~\ref{scattering_problem}.

In Section~2 we discuss the second member of the \PI\ 
hierarchy, equation \PItwo, as an equation that describes the 
isomonodromic solutions to the KdV equation, and find a nonlinear 
ODE, equation \PItwone, that controls its pole dynamics. In 
Section~3, we derive the above isomonodromic representation for 
the eq.\ \PItwone\ presenting a regularization of the linear system for 
\PItwo\ along its singularity locus. In Section~4, we consider 
the Puiseux series solution for \PItwone\ and the singularity 
reduction of the corresponding linear system at the branch points.
In section~5, we set the Riemann--Hilbert problems for all the
previously introduced wave functions and discuss existence and uniqueness
of their solutions. In Section~6, we present the asymptotic analysis
of the Riemann--Hilbert problems corresponding to a physically interesting
special solution of \PItwo\ and \PItwone\ implementing the steepest-descent 
method introduced by Deift and Zhou. Our main result in this section is
the description of the large $t$ asymptotics of the singularity locus
as a theta-divisor on a modulated elliptic curve,
\begin{equation*}
3x_0+\tfrac{2}{3}\xi_1(x_0)=
\tfrac{7}{4}t^{-7/4}
\bigl(
(n+\tfrac{1}{2})\omega_a(x_0)
+(m+\tfrac{1}{2})\omega_b(x_0)
\bigr),
\end{equation*}
where $n\in{\mathbb Z}$, $m\in{\mathbb Z}_+$ and $x_0=a(t)t^{-3/2}$, the function $\xi_1(x_0)$ is determined by equations \eqref{spectral_curve_genus1}--\eqref{Boutroux_tt}, $\omega_a(x_0)$ and $\omega_b(x_0)$ are the period integrals \eqref{AB_def} on the elliptic curve \eqref{surface_def}.
In Section~7, we find explicitly the pole dynamics of the special solution 
in a vicinity of the attracting point $x_0^*=\tfrac{2\sqrt5}{9\sqrt3}$,
\begin{multline}\label{pole_dynamics_sol}
a^{(m,n)}(t)=
\tfrac{2\sqrt5}{9\sqrt3}\,t^{3/2}
+t^{-1/4}(m+\tfrac{1}{2})\tfrac{\sqrt[4]{3}}{2\sqrt[4]{5}\sqrt{7}}
\ln\bigl(
t^{-7/4}(m+\tfrac{1}{2})\tfrac{3^{11/4}}{2^{6}5^{3/4}7^{5/2}e}
\bigr)
\\
+i\pi(n+\tfrac{1}{2})\tfrac{\sqrt[4]{3}}{\sqrt[4]{5}\sqrt{7}}
+{\mathcal O}(t^{-7/2}\ln^2t),
\end{multline}
where $m\in{\mathbb Z}_+$ and $n\in{\mathbb Z}$ enumerate
the points of the pole lattice. Thus, asymptotically for large $t\to+\infty$, the poles of the special solution to \PItwo ~ never collide. It would be interesting to study their collisions for finite time.
In Section~8, we briefly discuss our results and various open problems.

\section{Linear system system for \PItwo\ and its singularity locus}

\subsection{Linear system for \PItwo}

The 4th order 1st degree polynomial ODE,
\begin{equation}\label{p12}
u_{xxxx}
+10u_x^2
+20uu_{xx}
+40
\bigl(
u^3
-6tu
+6x
\bigr)
=0,
\end{equation}
usually denoted as \PItwo, is the second member of the so-called
\PI\ hierarchy. This equation governs the isomonodromy deformations
of a linear polynomial ODE described by the system 
\begin{equation}\label{Lax_pair_p12}
\Psi_{\lambda}=A\Psi,\quad
\Psi_x=B\Psi,\quad
\Psi_t=C\Psi,
\end{equation}
where the connection matrices are explicitly written
using the generators of ${\mathfrak s}{\mathfrak l}(2,{\mathbb C})$,
$\sigma_3=\left(\begin{smallmatrix}1&0\\0&-1\end{smallmatrix}\right)$, 
$\sigma_+=\left(\begin{smallmatrix}0&1\\0&0\end{smallmatrix}\right)$ and
$\sigma_-=\left(\begin{smallmatrix}0&0\\1&0\end{smallmatrix}\right)$,
\begin{subequations}\label{ABC_p12}
\begin{align}\label{A_p12}
\begin{split}
A=&\tfrac{1}{240}
\bigl\{
\bigl[
-4u_x\lambda-(12uu_x+u_{xxx})
\bigr]\sigma_3
\\
&\quad+\bigl[
8\lambda^2+8u\lambda+(12u^2+2u_{xx}-120t)
\bigr]\sigma_+
\\
&\quad+\bigl[
8\lambda^3-8u\lambda^2
-(4u^2+2u_{xx}+120t)\lambda
\\
&\hskip3cm
+(16u^3-2u_x^2+4uu_{xx}+240x)
\bigr]\sigma_-
\bigr\},
\end{split}
\\
\label{B_p12}
B=&\sigma_+
+(\lambda-2u)\sigma_-,
\\
\label{C_p12}
C=&\tfrac{1}{6}u_x\sigma_3
-\tfrac{1}{3}(\lambda+u)\sigma_+
-\tfrac{1}{3}(
\lambda^2
-u\lambda
-2u^2
-\tfrac{1}{2}u_{xx})\sigma_-,
\end{align}
\end{subequations}
Besides (\ref{p12}) which is the compatibility condition
$$
\left[ \partial_\lambda-A, \partial_x-B\right]=0
$$ 
of (\ref{A_p12}) and (\ref{B_p12}), the system (\ref{ABC_p12})
also implies the KdV equation,
\begin{equation}\label{KdV}
u_t+uu_x+\tfrac{1}{12}u_{xxx}=0,
\end{equation}
and equations which follow from (\ref{p12}) and (\ref{KdV}).
The solutions of \PItwo\ (\ref{p12}) compatible with the KdV equation
(\ref{KdV}) are called the {\em isomonodromic} \PItwo\ solutions to KdV.

\subsection{Laurent series solutions to \PItwo\ compatible with 
the KdV equation}

It is known that equation \PItwo\ passes Painlev\'e tests as being 
presented in the list by C.~Cosgrove \cite{Cos} under the symbol F-V
with $y=-u$ and the parameters $\alpha=240t$, $k=240$ and
$\beta=0$. See \cite{shimomura2} for a proof of the Painlev\'e property based on the Riemann--Hilbert correspondence.

Below, we are especially interested in the 4-parameter series 
solution to \PItwo\ with the following initial terms (because
\PItwo\ is polynomial in all its variables, the construction
of the complete formal series via a recurrence relation 
is straightforward),
\begin{multline}\label{Laurent_gen}
u(x)=-\frac{1}{(x-a)^2}
+\sum_{k=0}^{\infty}c_k(x-a)^k,
\\
c_1=0,\quad
c_2=3(c_0^2-2t),\quad
c_4=\tfrac{30}{7}a
-10c_0^3
+\tfrac{120}{7}c_0t,
\\
c_5=3-\tfrac{3}{2}c_0c_3,\quad
c_7=\tfrac{12}{7}(tc_3 - c_0),\quad\dots
\\
a,c_0,c_3,c_6\quad\mbox{are arbitrary}.
\end{multline}

This 4-parameter series is compatible with the KdV equation 
(\ref{KdV}) if the coefficients $c_0,c_3,c_6$ depend on $t$
in a particular way described by the function $a(t)$,
\begin{equation}\label{cj_via_a_in_KdV}
c_0=a',\quad
c_3=2a'',\quad
c_6=-\tfrac{1}{3}a'''
-3(a')^4
+12(a')^2t
-12t^2,\quad
(\ )'=\frac{d}{dt}.
\end{equation}
In its turn, the pole position $a(t)$ must satisfy the equation \PItwone,
\begin{equation}\label{a_ODE_from_KDV}
a^{(4)}
+120(a')^3a''
-120a'a''t
-\tfrac{200}{3}(a')^2
-\tfrac{40}{3}aa''
+\tfrac{200}{9}t=0.
\end{equation}
Besides the pole dynamics, the function $x=a(t)$ parametrizes 
the singularity locus for the linear system (\ref{Lax_pair_p12}), 
(\ref{ABC_p12}).

\begin{rem}
Equation \PItwo\ admits also the 3-parameter series solution,
\begin{multline}\label{Laurent_deg}
u(x)=
-\frac{3}{(x-\tilde a)^2}+\sum_{k=2}^{\infty}\tilde c_k(x-a)^k,
\\
\tilde c_2=-\tfrac{6}{7}t,\quad
\tilde c_3=0,\quad
\tilde c_4=-\tfrac{10}{21},\quad
\tilde c_5=-1,\quad
\tilde c_7=0,\quad
\tilde c_9=-\tfrac{50}{147}t,\quad
\\
\tilde c_{10}=\tfrac{2}{13}\bigl(
-\tfrac{50}{1323}\tilde a^2
+\tfrac{11}{7}\tilde c_6t
+\tfrac{36}{343}t^3
\bigr),\quad
\tilde c_{11}=
-\tfrac{20}{441}\tilde a,\quad
\dots,
\\
\tilde a,\tilde c_6,
\tilde c_8\quad
\mbox{are arbitrary},
\end{multline}
which, however, is not compatible with the KdV equation (\ref{KdV}). Clearly it corresponds to the triple collisions of the poles \eqref{Laurent_gen}.
\end{rem}

\section{Singularity reduction in the linear system for \PItwo\ 
and the isomonodromy property of \PItwone}

\subsection{Singularity reduction}
It is remarkable, that the linear system for \PItwo\ can be
regularized along the singularity locus $x=a(t)$:
\begin{thm}
Let $u(x)$ be the 4-parameter Laurent series solution
(\ref{Laurent_gen}) for equation {\em\PItwo}. Then the gauge transformation $Z=R\Psi$ with 
the gauge matrix
\begin{multline}\label{gauge_combined_def}
R(\lambda)=
2^{\frac{1}{2}\sigma_3}
(\lambda-2c_0)^{\frac{1}{2}\sigma_3}
(x-a)^{-\frac{1}{2}\sigma_3}
\tfrac{1}{\sqrt2}(\sigma_3+\sigma_1)
\times\hfill
\\
\times
(\lambda-2c_0)^{-\frac{1}{2}\sigma_3}2^{\frac{1}{2}\sigma_3}
(x-a)^{-\sigma_3}
\tfrac{1}{\sqrt2}(\sigma_3+\sigma_1)
(x-a)^{-\frac{1}{2}\sigma_3},\quad
\sigma_1=(\begin{smallmatrix}0&1\\1&0\end{smallmatrix}),
\end{multline}
regularizes the connections $\partial_{\lambda}-A$ and
$\partial_{x}-B$ along the singularity locus $x=a(t)$ and, if
$$
c_0=a'\quad\mbox{\rm and}\quad
c_3=2a'',
$$
it regularizes also the connection $\partial_t-C$ along the same locus.
\end{thm}
\begin{proof}
The proof is straightforward by inspection.
\end{proof}

\begin{rem}
The above gauge transformation is constructed as a sequence
of the so-called shearing transformations. Namely, substituting 
the initial terms of the relevant Laurent series instead of $u$ 
into the matrix $A$, one finds the leading order singular term
proportional to the sum 
$-(x-a)^{-4}\sigma_-+(x-a)^{-3}\sigma_3+(x-a)^{-2}\sigma_+$. 
Conjugation by $(x-a)^{-\sigma_3/2}$ ``shares'' the singularity 
order among the off-diagonal entries. The matrix coefficient of the 
term of order ${\mathcal O}((x-a)^{-3})$ becomes nilpotent, 
and its conjugation with 
$\tfrac{1}{\sqrt2}(\sigma_3+\sigma_1)$ turns it to the normal
$\sigma_-$-form. This brings into play the lower order singularities,
and the procedure can be repeated. 
\end{rem}

\begin{rem}
It is known that the singularity reductions exist also for 
the linear systems associated with the classical Painlev\'e equations.
However the construction of the corresponding gauge transformation 
does not rely on a sequence of the shearing transformations.
In contrast, the required gauge transformations simply turn the 
1st order rank 2 differential systems into scalar 2nd order
linear ODEs.
\end{rem}

\begin{conj}\label{singularity_reduction_existence}
The singularity reduction exists for an arbitrary isomonodromy system.
\end{conj}

\subsection{Isomonodromy representation for \PItwone}
The linear system for \PItwo\ regularized along its singularity locus $x=a(t)$
yields the linear system for equation \PItwone.

\begin{thm} Equation {\em\PItwone} controls the isomonodromy deformations
of a linear matrix ODE with the polynomial coefficients described by the system
\begin{equation}\label{a_Lax_pair}
Z_{\lambda}={\mathcal A}Z,\quad
Z_t={\mathcal C}Z,
\end{equation}
with the coefficient matrices
\begin{multline}\label{AC_a_Lax_pair}
{\mathcal A}=
-\tfrac{3}{20}a''\sigma_3
+\bigl(
\tfrac{1}{30}\lambda^4
+\alpha_3\lambda^3
+\alpha_2\lambda^2
+\alpha_1\lambda
+\alpha_0
\bigr)
\sigma_+
+(\tfrac{1}{30}\lambda-\alpha_3)\sigma_-,
\\
\alpha_3=-\tfrac{1}{10}a',\quad
\alpha_2=\tfrac{3}{10}(a')^2-t,\quad
\alpha_1=-\tfrac{9}{10}(a')^3+3a't+a,
\\
\alpha_0=3aa'
-\tfrac{27}{4}(a')^4
+\tfrac{9}{5} t (a')^2
+\tfrac{27}{20}c_6
+\tfrac{81}{5}t^2,
\\
\shoveleft{
{\mathcal C}=
\bigl[
-\tfrac{1}{3}\lambda^3
+\beta_2\lambda^2
+\beta_1\lambda
+\beta_0
\bigr]
\sigma_+
-\tfrac{1}{3}\sigma_-.
}\hfill
\\
\beta_2=2a',\quad
\beta_1=10t-9(a')^2,\quad
\beta_0=36(a')^3
-60a't
-10a.
\end{multline}
\end{thm}

\begin{proof}
By inspection, the compatibility condition, 
${\mathcal A}_t-{\mathcal C}_{\lambda}
+[{\mathcal A},{\mathcal C}]=0$, yields
\begin{equation*}
c_6=
-\tfrac{1}{3}a'''
-3(a')^4
+12(a')^2t
-12t^2,
\end{equation*}
and
\begin{equation*}
a''''
+120(a')^3a''
-120a'a''t
-\tfrac{200}{3}(a')^2
-\tfrac{40}{3}aa''
+\tfrac{200}{9}t
=0,
\end{equation*}
i.e.\ equations for $c_6$ in (\ref{cj_via_a_in_KdV})
and \PItwone\ (\ref{a_ODE_from_KDV}).
\end{proof}

\section{Puiseux series for $a(t)$ and the singularity reduction
of the linear system at the branch point of $a(t)$}

It is easy to verify that equation \PItwone\ admits
the formal series solution in powers of $(t-b)^{1/3}$
with the following initial terms (again, since the equation
is polynomial in all variables, the recurrence relation
for its coefficients is straightforward),
\begin{multline}\label{a_Puiseux_series}
a(t)=\sum_{k=0}^{\infty}
a_k(t-b)^{k/3},
\\
a_1=-1,\quad
a_2=a_3=a_4=0,\quad
a_5=-\tfrac{6}{7}{b},\quad
a_6=0,\quad
\\
a_7=\tfrac{10}{21}{a_0},\quad
a_8=-\tfrac{3}{2},\quad
a_{10}=0,\quad
a_{12}=-\tfrac{135}{49}{b},\quad
\dots,
\\
b,
a_0,
a_9,
a_{11}\quad
\mbox{are arbitrary}.
\end{multline}

The system (\ref{a_Lax_pair}), (\ref{AC_a_Lax_pair})
is singular, however regularizable, at the branch point $t=b$ 
of $a(t)$:
\begin{thm}
Let $a(t)$ be the 4-parameter Puiseux series solution
(\ref{a_Puiseux_series}) for equation \PItwone. Then
the rational in $\lambda$ and $t$ gauge transformation
$X=QZ$ with the gauge matrix
\begin{multline}\label{branch_point_gauge_combined}
Q(\lambda)=
\sigma_3
\bigl(
\tfrac{120}{7}b-\lambda^2
\bigr)^{-\frac{1}{2}\sigma_3}
(t-b)^{\frac{1}{6}\sigma_3}
\tfrac{1}{\sqrt2}(\sigma_3+\sigma_1)
\times
\\
\times
2^{-\sigma_3}
\lambda^{-\frac{1}{2}\sigma_3}
\bigl(
\tfrac{120}{7}b-\lambda^2
\bigr)^{\frac{1}{2}\sigma_3}
(t-b)^{\frac{1}{3}\sigma_3}
\tfrac{1}{\sqrt2}(\sigma_3+\sigma_1)
\times
\\
\times
\lambda^{\frac{1}{2}\sigma_3}
(t-b)^{\frac{1}{3}\sigma_3}
2^{-\frac{1}{2}\sigma_3}
\tfrac{1}{\sqrt2}(\sigma_3+\sigma_1)
(t-b)^{\frac{1}{2}\sigma_3}
\end{multline}
regularizes the connection $\partial_{\lambda}-{\mathcal A}$
at the branch point $t=b$.
\end{thm}

Here $\sigma_1$, $\sigma_2$, $\sigma_3$ are Pauli matrices,
$$
\sigma_1=\left( \begin{array}{cc} 0 & 1\\ 1 & 0\end{array}\right), \quad \sigma_2=\left(\begin{array}{cr} 0 & -i \\ i & 0\end{array}\right), \quad 
\sigma_3 =\left(\begin{array}{cr} 1 & 0\\ 0 & -1\end{array}\right).
$$ 

The new function $X(\lambda)$ satisfies the linear matrix ODE
equivalent to the anharmonic oscillator equation \eqref{oscil}, which does not 
admit any continuous isomonodromic deformation,
\begin{multline}\label{X_anharmonic_a_pole_p12}
X_{\lambda}X^{-1}=
\begin{pmatrix}
0&\frac1{30}\\
U(\lambda)&0
\end{pmatrix},
\\
U(\lambda)=
\tfrac{1}{30}\lambda^5
-b\lambda^3
+a_0\lambda^2
+(\tfrac{360}{49}b^2+\tfrac{33}{10}a_9)\lambda
-\tfrac{230}{21}a_0b
+\tfrac{143}{30}a_{11}.
\end{multline}

\begin{rem}\label{scattering_problem}
The anharmonic oscillator equation \eqref{oscil} can be regarded as a stationary 
Schr\"odinger equation. Thus it is possible to introduce a scattering 
problem and to map the set of the parameters $b$, $a_0$, $a_9$, 
$a_{11}$ to the set of the scattering data. Indeed, let us introduce 
the Jost solutions $\psi_{\pm}^{(k)}$ uniquely
determined by the asymptotics
\begin{multline*}
\psi_{\pm}^{(k)}\simeq
\lambda^{-\frac{5}{4}\sigma_3}
\tfrac{1}{\sqrt2}
\begin{pmatrix}
1\\\pm1
\end{pmatrix}
e^{\pm\theta},\quad
\theta=
\tfrac{1}{105}\lambda^{7/2}
-\tfrac{1}{3}t\lambda^{3/2}
+x\lambda^{1/2},
\\
\lambda\to\infty,\quad
\arg\lambda=\pi+\tfrac{2\pi}{7}k,\quad
k=0,\pm1,\pm2,\pm3.
\end{multline*}
Here, we define $\lambda^{1/2}$ on the $\lambda$ complex plane 
cut along the positive part of the real line and choose its principal 
branch. Then the incident wave $\psi\simeq\psi_-^{(0)}$ produces 
several reflected and transmitted waves,
\begin{equation*}
\psi=
t_k\psi_-^{(k)}+r_k\psi_+^{(k)},\quad
\arg\lambda=\pi+\tfrac{2\pi}{7}k,\quad
k=0,\pm1,\dots,\pm3,
\end{equation*}
where the coefficients $t_k$ and $r_k$ are called the transmission 
and reflection coefficients, respectively. Only $4$ of the $14$ 
coefficients are independent. Indeed, there exist
six relations of the form $r_{k+1}=r_k$, $t_k=t_{k-1}$, $k=-2,0,2$, 
which come from the preservation of the amplitude of the dominant 
solution in the sectors
$\arg\lambda\in[\pi+\tfrac{2\pi}{7}k,\pi+\tfrac{2\pi}{7}(k+1)]$,
$k=-3,\dots,2$; the equality $t_0=1$ normalizes the amplitude of 
the incident wave; two conditions $t_3=r_{-3}=0$ mean the solution 
is subdominant as $\lambda\to+\infty$; the last condition
$r_3=it_{-3}$ means the continuity of this subdominant
solution across the positive part of the real line,
\begin{multline}\label{tk_rk_rel}
t_3=r_{-3}=0,\quad
t_0=t_{-1}=1,\quad
t_2=t_1(=0),
\\
r_1=r_0(=i),\quad
r_{-1}=r_{-2}(=0),\quad
r_3=r_2=it_{-2}=it_{-3}(=i).
% \\
% t_3=0,\quad
% t_2=t_1=\tfrac{is_{-1}}{1+s_1s_2},\quad
% r_3=r_2=\tfrac{i}{1+s_1s_2},
% \\
% t_0=t_{-1}=1,\quad
% r_1=r_0=-\tfrac{s_1}{1+s_1s_2}-s_3,\quad
% r_{-1}=r_{-2}=-\tfrac{s_1}{1+s_1s_2},\quad
% \\
% t_{-2}=t_{-3}=\tfrac{1}{1+s_1s_2},\quad
% r_{-3}=0,
\end{multline}
The four independent scattering coefficients, say, 
$t_1,r_1,t_{-2},r_{-2}$, can be used to reconstruct 
the parameters $b$, $a_0$, $a_9$ and $a_{11}$ determining 
the potential $U(\lambda)$. 

In (\ref{tk_rk_rel}), the numbers in parentheses
specify the values of the free transmission and reflection 
coefficients corresponding to the triple collisions of the poles of the special 
solution to \PItwo\ of our interest, see below. Along the oscillatory
directions, the wave function corresponding to this special potential 
is given by the Jost solutions $\psi_-^{(k)}$ ($k=-1,-2,-3$), 
$i\psi_+^{(k)}$ ($k=1,2,3$) and by the sum $\psi_-^{(0)}+i\psi_+^{(0)}$ 
as $\lambda<0$. This implies another simple characterization of the 
potential of our interest.  Defining $\lambda^{1/2}$ on 
the plane cut along the negative part of the real axis, the Schr\"odinger 
equation $\psi''=\tfrac{1}{30}U(\lambda)\psi$ has the solution with 
the uniform in the sector $\arg\lambda\in(-\pi,\pi)$ asymptotics 
$\psi\simeq \lambda^{-5/4}e^{-\theta}$. Existence of a solution with such asymptotics is tantamount to vanishing of four Stokes multipliers $s_{\pm 1}=s_{\pm 2}=0$ (see below).

Below, however, we characterize $X(\lambda)$ adopting the Riemann-Hilbert 
problem instead of the scattering problem.
\end{rem}

\begin{rem}
Observe the successive reduction of the number of the deformation 
parameters the $\lambda$-equations depend on: for the connection 
$\partial_{\lambda}-A$, the space of the deformation parameters 
$(t,x)$ is 2-dimensional, for $\partial_{\lambda}-{\mathcal A}$, 
the deformation parameter space reduces to the 1-di\-men\-si\-on\-al 
space of $t$, and for $\partial_{\lambda}-\tfrac{1}{30}\sigma_+-U\sigma_-$, 
it is 0-dimensional. 
\end{rem}

\begin{rem}
Besides the continuous isomonodromic deformations considered above, 
it is possible to introduce discrete Poincar\'e-like isomonodromy mappings 
from one to another branch of $a(t)$ and from one to another point of 
the lattice $(b,a(b))$.
\end{rem}

\begin{rem}
It is natural to denote the above listed connections by the symbols 
\LItwotwo, \LItwone and \LItwozero, respectively, to reflect explicitly 
the number of the continuous deformation parameters involved. 
The equations of the isomonodromic deformations are thus denoted by
the symbols \PItwotwo\ (same as \PItwo), \PItwone
and \PItwozero (the latter equation is a Poincar\'e-like mapping
of the lattice $(b,a(b))$). 
\end{rem}

\section{Riemann--Hilbert problems for \LItwotwo, \LItwone 
and \LItwozero}

The boundary Riemann--Hilbert (RH) problem consists in finding a piece-wise 
holomorphic function by its prescribed analytic properties:

i) asymptotics at some marked points;

ii) discontinuity properties across a piece-wisely oriented 
graph.

\subsection{Asymptotics of the canonical solutions 
$\Psi_k(\lambda)$, $Z_k(\lambda)$ and $X_k(\lambda)$}

Each of the above linear matrix ODEs for $\Psi$, $Z$ and $X$ has one
irregular singularity at $\lambda=\infty$ and no other singular
points. We introduce the formal series solutions that represent 
the asymptotic behavior of the genuine solutions to these linear ODEs
in the interior of particular sectors near $\lambda=\infty$.

Since all the canonical asymptotics differ from each other in a rational 
left diagonal multiplier, it is convenient to unify the notations
and distinguish all three cases introducing a parameter $\nu$:
\begin{multline}\label{nu_canonical_as}
\Phi_k^{(\nu)}(\lambda)=
\lambda^{\frac{\nu}{4}\sigma_3}
\tfrac{1}{\sqrt2}(\sigma_3+\sigma_1)
(I+{\mathcal O}(\lambda^{-1/2}))
e^{\theta\sigma_3},
\\
\shoveleft
\theta=\tfrac{1}{105}\lambda^{7/2}
-\tfrac{1}{3}t\lambda^{3/2}
+x\lambda^{1/2},
\\
\lambda\to\infty,\quad
\arg\lambda\in
(-\tfrac{3\pi}{7}+\tfrac{2\pi}{7}k,
\tfrac{\pi}{7}+\tfrac{2\pi}{7}k),\quad
k\in{\mathbb Z},
\\
\nu\in\{-1,3,-5\},\quad
\Phi^{(-1)}=\Psi,\quad
\Phi^{(3)}=Z,\quad
\Phi^{(-5)}=X.
\hfill
\end{multline}
Here, the principal branch of the square root of $\lambda$ is chosen.
\begin{rem}
Using (\ref{nu_canonical_as}) for the asymptotic solutions 
$Z_k$ and $X_k$, one has to take into account the reductions 
$x=a(t)$ for $\nu=3$ and $t=b$, $x=a(b)$ for $\nu=-5$. 
Below however, using  (\ref{nu_canonical_as}) to set the corresponding
RH problems, we assume that the exponential $\theta$ depends on
arbitrary complex deformation parameters $t$ and $x$
because it is not allowed to use the unknown values $a(t)$ and 
$b,a(b)$ as the data in the RH problems formulated below.
As the result, it is necessary to remember that the solvability domains for
the RH problems shrink to the lines $x=a(t)$ for $\nu=3$ and 
to a lattice $(t,x)=(b,a(b))$ for $\nu=-5$.
\end{rem}

\subsection{Stokes multipliers}

The canonical solutions differ from each other in a triangular right
matrix multipliers called the Stokes matrices,
\begin{equation}\label{Sk_def}
\Phi_{k+1}(\lambda)=\Phi_k(\lambda)S_k,\quad
S_{2k-1}=
\begin{pmatrix}
1&s_{2k-1}\\
0&1
\end{pmatrix},\quad
S_{2k}=
\begin{pmatrix}
1&0\\
s_{2k}&1
\end{pmatrix},
\end{equation}
moreover
\begin{equation}\label{Sk_struct}
s_{k+7}=s_k,\quad
s_k+s_{k+2}+s_ks_{k+1}s_{k+2}=-i(1+s_{k+4}s_{k+5}),\quad
k\in{\mathbb Z}.
\end{equation}

\begin{rem}
Since the gauge transformations $R(\lambda)$ and $Q(\lambda)$ 
are rational, they do not affect the Stokes matrices.
\end{rem}

\begin{rem}
The Stokes multipliers $s_k$ are the first integrals of equations 
\PItwotwo\ (\ref{p12}) and \PItwone\ (\ref{a_ODE_from_KDV}).
\end{rem}

The solutions of the RH problems below are constructed using collections 
of the canonical solutions in the above presented sectors
in a way to have the uniform asymptotics in a vicinity of infinity.
The normalization of the RH problem however is rather subtle thing.
To this aim, observe the prolongated canonical asymptotics
\begin{equation}\label{Psi_canonical_as_long}
\tfrac{1}{\sqrt2}(\sigma_3+\sigma_1)
\lambda^{\frac{1}{4}\sigma_3}
\Psi_k(\lambda)
e^{-\theta\sigma_3}=
(I+\tfrac{u}{2\lambda}\sigma_1
+{\mathcal O}(\lambda^{-3/2}))
e^{\lambda^{-1/2}d_1\sigma_3},
\end{equation}
\begin{multline}\label{Z_canonical_as_long}
\tfrac{1}{\sqrt2}(\sigma_3+\sigma_1)
\lambda^{-\frac{3}{4}\sigma_3}
Z_k(\lambda)
e^{-\theta\sigma_3}
=
\\
=(
I
-\tfrac{3a'}{2\lambda}
\sigma_1
+\tfrac{1}{\lambda^2}(\tfrac{9}{4}(a')^2-\tfrac{15}{2}t)
\sigma_1
+{\mathcal O}(\lambda^{-5/2})
)
\times
\\
\times
\exp\{\lambda^{-1/2}
d_1\sigma_3
+\lambda^{-3/2}
d_3\sigma_3
+\lambda^{-2}
d_4I\},
\end{multline}
and
\begin{multline}\label{X_canonical_as_long}
\tfrac{1}{\sqrt2}(\sigma_3+\sigma_1)
\lambda^{\frac{5}{4}\sigma_3}
X_k(\lambda)
e^{-\theta\sigma_3}
=(I+\tfrac{15b}{2\lambda^2}\sigma_1
-\tfrac{15a_0}{2\lambda^3}\sigma_1
+{\mathcal O}(\lambda^{-7/2}))
\times
\\
\times
\exp\{\lambda^{-1/2}d_1\sigma_3
+\lambda^{-3/2}d_3\sigma_3
+\lambda^{-5/2}d_5\sigma_3\}.
\end{multline}
In the above right hand side exponential factors,
the parameters $d_j$ can be expressed in terms 
of the coefficients of the relevant equation. 
For instance, $d_1$ in (\ref{Psi_canonical_as_long}) is
one of two Hamiltonians of \PItwo. However, these 
particular relations are not important for us at this stage. 

Now, we have the following

\begin{RHP}\label{nu_RHP}
Given the complex values of the parameters $x$, $t$ and $s_k$, 
$k\in{\mathbb Z}$, satisfying (\ref{Sk_struct}), and denoting
\begin{equation}\label{theta_Phi_def}
\theta=
\tfrac{1}{105}\lambda^{7/2}
-\tfrac{1}{3}t\lambda^{3/2}
+x\lambda^{1/2},
\end{equation}
find the piece-wise holomorphic $2\times2$ matrix function 
$\Phi^{(\nu)}(\lambda)$, $\nu\in\{-1,3,-5\}$, with the 
following properties:
\begin{enumerate}
\item 
\begin{equation}\label{Phi_as_def}
\lim_{\lambda\to\infty}
\tfrac{1}{\sqrt2}(\sigma_3+\sigma_1)
\lambda^{-\frac{\nu}{4}\sigma_3}
\Phi^{(\nu)}(\lambda)
e^{-\theta\sigma_3}=I,
\end{equation}
moreover
\begin{itemize}
\item
there exist a constant $d_1$ such that
\begin{multline}\label{Phi_c0_triviality_condition}
\hskip0.72in
\tfrac{1}{\sqrt2}(\sigma_3+\sigma_1)
\lambda^{-\frac{\nu}{4}\sigma_3}\Phi^{(\nu)}(\lambda)
e^{-\theta\sigma_3}=
\\
=I
+\lambda^{-1/2}d_1\sigma_3
+{\mathcal O}(\lambda^{-1});
\end{multline}

\item
if $\nu\in\{3,-5\}$, then there exist
constants $d_1$, $d_3$ such that
\begin{multline}\label{Z_c1_triviality_condition}
\hskip0.72in
\tfrac{1}{\sqrt2}(\sigma_3+\sigma_1)
\lambda^{-\frac{3}{4}\sigma_3}
2^{\frac{1}{2}\sigma_3}
\Phi^{(\nu)}(\lambda)
e^{-(\theta+d_1\lambda^{-1/2})\sigma_3}=
\\
=I+\lambda^{-1}c_1
+\lambda^{-3/2}d_3\sigma_3
+{\mathcal O}(\lambda^{-2})
\end{multline}
with some constant matrix $c_1$;
\item
if $\nu=-5$, then there exist constants $d_1,d_3$ and $d_5$ such that
\begin{multline}\label{X_c2_triviality_condition}
\hskip0.72in
\tfrac{1}{\sqrt2}(\sigma_3+\sigma_1)
\lambda^{-\frac{3}{4}\sigma_3}
2^{\frac{1}{2}\sigma_3}
\Phi^{(\nu)}(\lambda)
e^{-(\theta+d_1\lambda^{-1/2}+d_3\lambda^{-3/2})\sigma_3}=
\\
=I+\lambda^{-2}c_2
+\lambda^{-5/2}d_5\sigma_3
+{\mathcal O}(\lambda^{-3})
\end{multline}
with some constant matrix $c_2$;
\end{itemize}
\item
$\|\Phi^{(\nu)}(\lambda)\|<const$ as $\lambda\to0$;
\item 
on the union of eight rays 
$\gamma=\rho\cup\bigl(\cup_{k=1}^7\gamma_{k-4}\bigr)$, where
$\gamma_k=\bigl\{\lambda\in{\mathbb C}\colon
\arg\lambda=\tfrac{2\pi}{7}k\bigr\}$, $k=-3,-2,\dots,2,3$, and 
$\rho=\bigl\{\lambda\in{\mathbb C}\colon
\arg\lambda=\pi\bigr\}$, all oriented towards infinity,
the jump condition holds true,
\begin{equation}\label{jump_rel}
\Phi^{(\nu)}_+(\lambda)=\Phi^{(\nu)}_-(\lambda)S(\lambda),
\end{equation}
where $\Phi^{(\nu)}_+(\lambda)$ and $\Phi^{(\nu)}_-(\lambda)$ are 
limits of $\Phi^{(\nu)}(\lambda)$ on $\gamma$ from the left and 
from the right, respectively, and the piece-wise constant matrix 
$S(\lambda)$ is given by equations
\begin{subequations}\label{jump_matrices}
\begin{align}\label{Sk_jumps}
&S(\lambda)\bigr|_{\lambda\in\gamma_k}=S_k,\quad
S_{2k}=I+s_{2k}\sigma_-,\quad
S_{2k-1}=I+s_{2k-1}\sigma_+,
\\
\label{sigma1_jump}
&S(\lambda)\bigr|_{\rho}=i\sigma_1.
\end{align}
\end{subequations}
\end{enumerate}
\end{RHP}

%%%%%%%%%%%%%%%%%%%%%%%%%%%%
\begin{figure}[htb]
\begin{center}
\mbox{\epsfig{figure=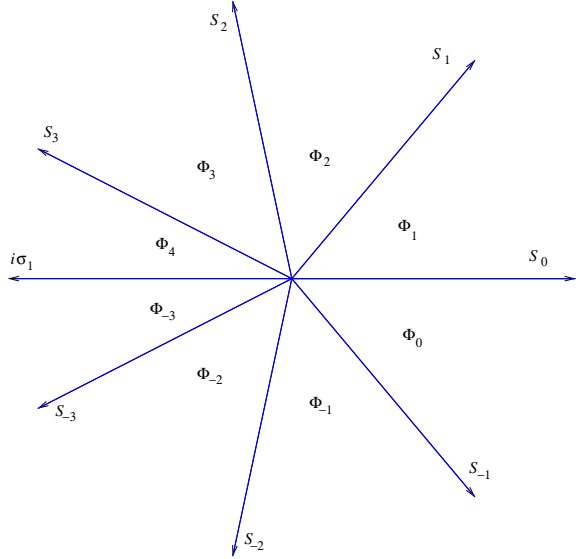,width=0.6\textwidth}}
\end{center}
\caption{The jump contour $\gamma$ for the RH problem~\ref{nu_RHP}
and canonical solutions $\Phi_j(\lambda)$, $j=-3,-2,\dots,3,4$.}
\label{fig1}
\end{figure}
%%%%%%%%%%%%%%%%%%%%%%%%%%%%%

\subsubsection{Uniqueness of the solution to the RH problem~\ref{nu_RHP}}

The solution of the RH problem~\ref{nu_RHP} is unique.
Indeed, the scalar function $\det\Phi^{(\nu)}(\lambda)$ is continuous 
across the set of rays $\gamma$ and is bounded at the origin. Therefore, 
$\det\Phi^{(\nu)}(\lambda)$ is an entire function, and taking into account 
(\ref{Phi_as_def}) and applying the Liouville theorem, 
$\det\Phi^{(\nu)}(\lambda)\equiv-1$. Assume for the moment that there 
are two solutions of the RH problem~\ref{nu_RHP}, $\Phi^{(\nu)}(\lambda)$ 
and $\tilde\Phi^{(\nu)}(\lambda)$. Consider their ratio, 
$\chi(\lambda)=
(
\begin{smallmatrix}
a&b\\
c&d
\end{smallmatrix}
)
=\tilde\Phi^{(\nu)}(\lambda)(\Phi^{(\nu)}(\lambda))^{-1}$.
As it is easy to see, $\chi(\lambda)$ is continuous across all 
the rays of the set $\gamma$ and is bounded at the origin, thus
$\chi(\lambda)$ is an entire function. Using (\ref{Phi_as_def}),
\begin{equation*}
\lim_{\lambda\to\infty}
\lambda^{-\frac{\nu}{4}\sigma_3}
\chi(\lambda)
\lambda^{\frac{\nu}{4}\sigma_3}
=\lim_{\lambda\to\infty}
\begin{pmatrix}
a&b\lambda^{-\nu/2}\\
c\lambda^{\nu/2}&d
\end{pmatrix}
=I.
\end{equation*}
Since all the entries $a,b,c,d$ are entire functions in $\lambda$, 
the Liouville theorem yields the ambiguity 
\begin{multline}\label{tilde_Psi_Psi_def}
\Phi^{(\nu)}(\lambda)\mapsto
\tilde\Phi^{(\nu)}(\lambda)=
P^{(\nu)}(\lambda)
\Phi^{(\nu)}(\lambda),
\\
P^{(-1)}(\lambda)=
I+c_0\sigma_-,\quad
P^{(3)}(\lambda)=
I+(c_0\lambda+c_1)\sigma_+,
\\
P^{(-5)}(\lambda)=
I+(c_0\lambda^2+c_1\lambda+c_2)\sigma_-,
\end{multline}
where $c_j$ are arbitrary constants.
These ambiguities however are eliminated using the asymptotic conditions
in (\ref{Phi_c0_triviality_condition}), (\ref{Z_c1_triviality_condition})
and (\ref{X_c2_triviality_condition}).

\subsection{Solvability of the RH problems~\ref{nu_RHP}
and the Malgrange divisor}

Given jump matrices, the set of points $(x,t)$ at which 
the RH problem is not solvable is called the Malgrange divisor. 
It coincides with the zero locus of a holomorphic $\tau$-function
by Miwa and with the singularity locus of the isomonodromy 
deformation equation.

For what follows, it is convenient to articulate our assumptions
on the singularities and critical points of equations \PItwo\ 
and \PItwone.
\begin{conj}\label{PP_conjecture}
1)~Equation \PItwotwo\ has no other movable singularities except for the
movable poles (\ref{Laurent_gen}) satisfying (\ref{cj_via_a_in_KdV}), 
(\ref{a_ODE_from_KDV}), and their triple collisions (\ref{Laurent_deg}).
2)~Equation \PItwone\ has no other movable singularities or critical points
except for the branch points (\ref{a_Puiseux_series}).
\end{conj}

These assumptions mean that the smooth branches of the Malgrange divisor 
for the RH problem~\ref{nu_RHP} with $\nu=-1$ (corresponding to \PItwotwo)
are parameterized by equations $x=a(t)$ and thus coincide with the solvability
set for the RH problem~\ref{nu_RHP} with $\nu=3$ (corresponding to \PItwone).
Similarly, the vertices of the Malgrange divisor of the RH problem with $\nu=-1$ 
correspond to the branch points $t=b$ of $a(t)$ and therefore coincide with 
the solvability set of the RH problem with $\nu=-5$.

In other words, Conjecture~\ref{PP_conjecture} implies the following
remarkable properties of the domains of solvability of all three RH problems:
\begin{enumerate}
\item
these domains do not pairwise intersect;
\item
all these domains together cover the deformation parameter space 
${\mathbb C}\times{\mathbb C}\ni(t,x)$.
\end{enumerate}
%%%%%%%%%%%%%%%%%%%%%%%%%%%%
\begin{figure}[htb]
\begin{center}
\mbox{\epsfig{figure=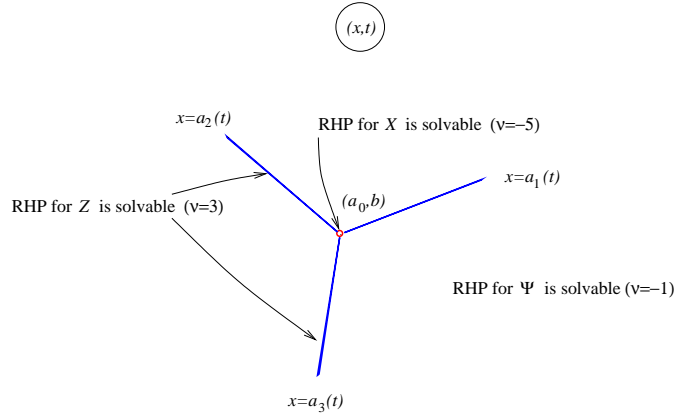,width=0.7\textwidth}}
\end{center}
\caption{The scheme of the Malgrange divisor and the domains
of the RH problems solvability for $\nu=-1,3,-5$.}
\label{fig1a}
\end{figure}
%%%%%%%%%%%%%%%%%%%%%%%%%%%%%

\subsection{Normalized non-homogeneous and 
homogeneous RH problems}

In this section, we describe another interesting feature of the family
of three RH problems. Roughly speaking, any solution of the 
{\em non-homogeneous} RH problem with the bigger value of $|\nu|$ 
allows one to construct infinitely many solutions of the {\em homogeneous}
problems with the smaller value of $|\nu|$.

Precise formulation of this property requires some more
accuracy. First of all, let us introduce the RH problems
equivalent to those above but unbranched 
and normalized at infinity. To this aim, write
\begin{equation}\label{zeta4_hat_Phi_transform}
\lambda=\zeta^4,\quad
\hat\Phi^{(\nu)}_k(\zeta)=
\tfrac{1}{\sqrt2}(\sigma_3+\sigma_1)
\zeta^{-4\nu\sigma_3}
\Phi^{(\nu)}_k(\zeta^4)
e^{-\theta(\zeta^4)},
\end{equation}
and define the piece-wise holomorphic functions,
\begin{equation}\label{hat_Psi_def}
\hat\Phi^{(\nu)}(\zeta)=
\hat\Phi^{(\nu)}_k(\zeta),\quad
\arg\zeta\in\bigl(
\tfrac{\pi}{14}(k-1),\tfrac{\pi}{14}k
\bigr),\quad
\nu=-1,3,-5.
\end{equation}
These functions solve the non-homogeneous RH problems 
on the union of rays $\cup_{k=-13}^{14}\ell_k$, 
$\ell_k=\{\zeta\in{\mathbb C}\colon\
\arg\zeta=\tfrac{\pi}{14}k\}$ 
with a singular point at the origin and normalized 
to the unit at infinity. The corresponding homogeneous 
RH problem differs from the non-homogeneous counterpart
in the asymptotics of $\Phi^{(\nu)}(\zeta)$ that vanishes 
as $\zeta\to\infty$.

Slightly abusing our notations, we formulate the non-homogeneous
and homogeneous RH problems as follows:

\begin{RHP}\label{hat_Phi_RHP}
Find a piece-wise holomorphic function 
$\hat\Phi_{I}^{(\nu)}(\zeta)$ 
$(\mbox{resp.,}\ \hat\Phi_{0}^{(\nu)}(\zeta))$, 
$\nu\in\{-1,3,-5\}$,
with the following properties:
\begin{enumerate}
\item
$\lim_{\zeta\to\infty}\hat\Phi_{I}^{(\nu)}(\zeta)=I$\quad
$(\mbox{resp.,}\quad\lim_{\zeta\to\infty}\hat\Phi_{0}^{(\nu)}(\zeta)=0)$;
\item
across the rays $\ell_k=\{\zeta\in{\mathbb C}\colon
\arg\zeta=\tfrac{\pi}{14}k\}$, $k=-13,\dots,13,14$,
all oriented towards infinity, the following jump conditions hold,
\begin{equation}\label{hat_Phi_jump_infty}
\hat\Phi^{(\nu)}_+(\zeta)=
\hat\Phi^{(\nu)}_-(\zeta)
e^{\theta\sigma_3}S_ke^{-\theta\sigma_3},\quad
\zeta\in\ell_k,\quad
\theta=\tfrac{1}{105}\zeta^{14}
-\tfrac{1}{3}t\zeta^{6}
+x\zeta^2;
\end{equation}
\item
$\|\zeta^{\nu\sigma_3}
\tfrac{1}{\sqrt2}(\sigma_3+\sigma_1)
\hat\Phi^{(\nu)}(\zeta)\|<const$,\quad
as $\zeta\to0$.
\end{enumerate}
\end{RHP}

The following theorem holds true
\begin{thm}
Any solution $\hat\Phi_{I}^{(3)}(\zeta)$ 
to the non-homogeneous RH problem~\ref{hat_Phi_RHP} 
with $\nu=3$ yields infinitely many solutions 
$\hat\Phi_{0}^{(-1)}(\zeta)$ to the homogeneous 
RH problem~\ref{hat_Phi_RHP} with $\nu=-1$.
Any solution $\hat\Phi_{I}^{(-5)}(\zeta)$ to 
the non-homogeneous RH problem~\ref{hat_Phi_RHP} 
with $\nu=-5$ yields infinitely many solutions 
$\hat\Phi_{0}^{(\nu)}(\zeta)$ to the homogeneous 
RH problem~\ref{hat_Phi_RHP} with $\nu=3,-1$.
\end{thm}
\begin{proof}
We prove the first part of the theorem. The proof of 
the second part is similar.

Consider a solution $\hat\Phi_{I}^{(3)}(\zeta)$ to the RH 
problem~\ref{hat_Phi_RHP} for $\nu=3$. It has the asymptotics
\begin{multline*}
\hat\Phi_{I}^{(3)}(\zeta)=I+{\mathcal O}(\zeta^{-2}),\quad
\zeta\to\infty,
\\
\shoveleft{
\hat\Phi_{I}^{(3)}(\zeta)=
\tfrac{1}{\sqrt2}(\sigma_3+\sigma_1)
\zeta^{-3\sigma_3}
\bigl(
M_0+{\mathcal O}(\zeta^4)
\bigr),\quad
\zeta\to0,\quad
\zeta\in\omega_0,
}\hfill
\end{multline*}
where $M_0$ is a constant matrix, $\det M_0\neq0$, and 
$\omega_0=\{\zeta\in{\mathbb C}\colon
\arg\zeta\in(0,\tfrac{\pi}{14})\}$.
Then
\begin{multline*}
\hat\Phi_0^{(-1)}(\zeta)=\zeta^{-2}N\hat\Phi_{I}^{(3)}(\zeta),
\\
N=\tfrac{1}{\sqrt2}(\sigma_3+\sigma_1)
\begin{pmatrix}
0&p\\
0&q
\end{pmatrix}
\tfrac{1}{\sqrt2}(\sigma_3+\sigma_1),\quad
p,q=const,
\end{multline*}
has the same jump properties as $\hat\Phi_{I}^{(3)}(\zeta)$ 
and the asymptotics of the form,
\begin{multline*}
\hat\Phi_0^{(-1)}(\zeta)={\mathcal O}(\zeta^{-2}),\quad
\zeta\to\infty,
\\
\shoveleft{
\hat\Phi_0^{(-1)}(\zeta)=
\tfrac{1}{\sqrt2}(\sigma_3+\sigma_1)
\zeta^{\sigma_3}
\begin{pmatrix}
0&p\\
0&q\zeta^2
\end{pmatrix}
\bigl(
M_0
+{\mathcal O}(\zeta^4)
\bigr),\quad
\zeta\to0,\quad
\zeta\in\omega_0.
}\hfill
\end{multline*}
Thus $\hat\Phi_0^{(-1)}(\zeta)$
is the solution of the homogeneous RH 
problem~\ref{hat_Phi_RHP} with $\nu=-1$.
\end{proof}

\section{Large $t$ asymptotics of a special solution
of equations \PItwo\ and \PItwone}

In this section, we construct large $t$ asymptotic 
solution to the RH problem~\ref{nu_RHP} for $\nu=-1$ 
and $\nu=3$ corresponding to a special solution of 
equation \PItwo. This special solution $u(x,t)$ has 
the asymptotics $u\sim\mp\sqrt[3]{6|x|}$ 
as $x\to\pm\infty$ and is real and regular on the real 
line for any $t\in\mathbb R$. The physical importance of this solution for $t=0$ from the point
of view of string theory was justified in \cite{BMP}.
In \cite{D1, D2}, this solution appeared in the study of the problem of universality of
critical behavior of solutions to Hamiltonian perturbations of hyperbolic PDEs.

The above mentioned properties uniquely distinguish this
special solution. In \cite{K}, the characterization of this 
solution in terms of the Stokes multipliers of 
the associated linear system was found,
\begin{equation}\label{tri-truncated_Stokes}
s_{-2}=s_{-1}=s_1=s_2=0,\quad
s_{-3}=s_0=s_3=-i.
\end{equation}

\subsection{Large $t$ asymptotic spectral curve 
for the special solution to \PItwo}

Assume that $t>0$ is large and $x\in{\mathbb C}$ is such that
\begin{equation}\label{t_x_ratio}
x_0:=xt^{-3/2}={\mathcal O}(1)\quad\mbox{as}\quad
t\to+\infty.
\end{equation}

Our starting point is the large $t$ asymptotics
of the spectral curve, $\det(\mu-t^{-5/4}A(\lambda))=0$,
where $\lambda=t^{1/2}\xi$, i.e.
\begin{equation}\label{spectral_curve_p12}
\mu^2=
\tfrac{1}{900}\xi^5
-\tfrac{1}{30}\xi^3
+\tfrac{1}{30}x_0\xi^2
+\tfrac{1}{30}D_1\xi
+\tfrac{1}{30}D_0=
\tfrac{1}{900}
\prod_{k=1}^5(\xi-\xi_k).
\end{equation}
The asymptotic analysis of various degenerated solutions 
of \PItwo\ performed in \cite{GKK}, shows that the topological
properties of the asymptotic spectral curve are significantly 
different in the interior of the domains $D_{\pm}$ and $D_0$, 
see Figure~\ref{fig3a}. For $x_0\in D_{\pm}$, the spectral curve 
has genus~0,
\begin{equation}\label{spectral_curve_genus0}
\mu^2=
\tfrac{1}{900}
(\xi-\xi_i)^2
(\xi-\xi_j)^2
(\xi-\xi_k),
\end{equation}
where the double branch points $\xi_i,\xi_j$ and the simple 
branch point $\xi_k$ satisfy the conditions,
\begin{equation*}
\xi_{i,j}=
-\tfrac{1}{4}\xi_k
\pm\sqrt{15-\tfrac{5}{16}\xi_k^2},\quad
\xi_k^3
-24\xi_k
+48x_0=0,
\end{equation*}
and the ambiguity in choice of the root of the cubic equation
for the simple branch point $\xi_k$ is fixed demanding that, 
for $x_0\in{\mathbb R}$,
\begin{equation*}
\sgn(x_0)\,\Re\Bigl(
\int_{\xi_k}^{\xi_{i,j}}
\mu(z)\,dz
\Bigr)>0,\quad
x_0\to\pm\infty.
\end{equation*}
The conditions above are consistent with the quasi-stationary 
asymptotic behavior of the special solution to \PItwo\ (\ref{p12}),
\begin{multline}\label{quasi-stationary_sol_PI2}
u(x,t)\simeq t^{1/2}v_0,\quad
v_0^3-6v_0+6x_0=0,
\\
v_0\simeq-\sqrt[3]{6x_0},\quad
x_0\to\pm\infty,
\end{multline}
where the real on the real line branch of $\sqrt[3]{6x_0}$
is chosen.

%%%%%%%%%%%%%%%%%%%%%%%%%%%%
\begin{figure}[htb]
\begin{center}
\mbox{\epsfig{figure=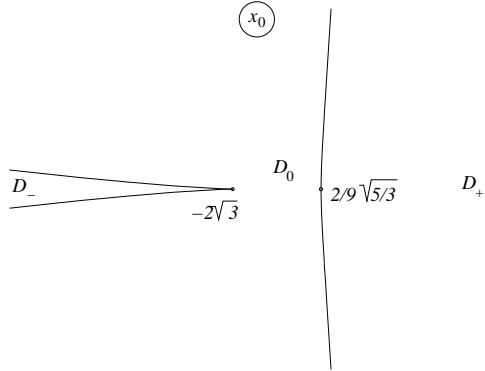,width=0.5\textwidth}}
\end{center}
\caption{The discriminant set for the special solution 
of equation \PItwo. In the interior of the domains $D_+$ and $D_-$, 
the asymptotics of the special solution has genus $g=0$, and in 
the domain $D_0$, it has genus $g=1$.}
\label{fig3a}
\end{figure}
%%%%%%%%%%%%%%%%%%%%%%%%%%%%%

Let us indicate some interesting points in the $x_0$ complex plane.

At the point $x_0=-2\sqrt3=D_0\cap D_-\cap{\mathbb R}$, 
two double branch points of the asymptotic spectral curve 
coalesce, so that $\xi_{1,2,3,4}=-\sqrt3$ and $\xi_5=4\sqrt3$.
Generically, at the asymptotically quadruple branch point, the local 
solution of the RH problem can be approximated using the Garnier--Jimbo--Miwa 
$\Psi$ function for the second Painlev\'e transcendent \PII, \cite{Gar, JMU}. 
For the RH problem we are studying, the relevant local approximate solution
corresponds to a Hastings--McLeod solution to \PII\ \cite{CG2}.

At any $x_0\in D_+$ including its boundary, the spectral curve has one 
simple and two double branch points. The precise asymptotic location of these 
branch points can be found for 
$x_0=\tfrac{2\sqrt5}{9\sqrt3}=D_0\cap D_+\cap{\mathbb R}$, where 
$\xi_3=-\tfrac{4\sqrt5}{\sqrt3}$, $\xi_{1,2}=-\tfrac{\sqrt5}{\sqrt3}$
and $\xi_{4,5}=\tfrac{3\sqrt5}{\sqrt3}$. At this point, as well as at 
any point of $D_+$ including its boundary, the leading order asymptotic
solution of the RH problem can be expressed in elementary functions, 
cf.\ \cite{Cl} for the real line case.

Observe also the point $x_0=2\sqrt3\in D_+$ where two double branch
points coalesce. The asymptotic branch points corresponding to this point 
are $\xi_3=-4\sqrt3$ and $\xi_{1,2,4,5}=\sqrt3$. Generically, the quadruple 
degeneration corresponds to the appearance of \PII. However, for the RH 
problem we consider here, the relevant Painlev\'e function is trivial, and 
the asymptotic solution to the RH problem remains elementary. 

In the interior part of the domain $D_0$, the large $t$ asymptotic solution 
to the RH problem is constructed on the model elliptic curve,
\begin{equation}\label{spectral_curve_genus1}
\mu^2=\tfrac{1}{900}
(\xi-\xi_1)^2
(\xi-\xi_3)(\xi-\xi_4)(\xi-\xi_5),
\end{equation}
where the branch points $\xi_j$, $j=3,4,5$, are determined
by the values of $\xi_1$ and $x_0$ as the roots of the cubic equation
\begin{equation}\label{xi345_via_xi1_eq}
\xi^3
+2\xi_1\xi^2
+(3\xi_1^2-30)\xi
+4\xi_1^3
-60\xi_1
+30x_0=0.
\end{equation}
The double branch point $\xi_1$ is determined
as a function of $x_0\in D_0$ by the system of Boutroux 
equations, see \cite{K2},
\begin{equation}\label{Boutroux_tt}
\Re\int_{\xi_3}^{\xi_4}
\mu(z)\,dz=0,\quad
\Re\int_{\xi_4}^{\xi_5}
\mu(z)\,dz=0,
\end{equation}
supplemented by the described above boundary conditions 
on $D_{\pm}\cap D_0$.

\begin{rem}
As $x_0$ approaches the real segment 
$(-2\sqrt3,\tfrac{2\sqrt5}{9\sqrt3})$,
all asymptotic branch points become real and satisfy the
inequalities
$\xi_3<\xi_1<\xi_4<\xi_5$. Thus the second of the 
equations (\ref{Boutroux_tt}) trivializes while 
the first of these equations turns into the condition 
$\int_{\xi_3}^{\xi_4}\mu(z)\,dz=0$
obtained in \cite{Pot} in the analysis of the Whitham 
equations and used in \cite{Cl} to study the same special 
solution on the real line.
\end{rem}

\subsection{Steepest-descent analysis of the RH problem}

The strategy of the steepest-descent asymptotic analysis
by Deift and Zhou \cite{DZ1, DZ2} of the RH problem involves 
several standard steps: 
1)~transformation of the jump graph to the steepest-descent
directions of a suitable $g$-function; 
2)~construction of the local approximate solutions (parametrices); 
3)~matching all the parametrices into a global parametrix;
4)~proof that the global parametrix indeed approximates
the genuine solution to the original RH problem.

Since all the above mentioned steps are well explained in the 
literature (see e.g.\ \cite{FIKN}), below, we omit unnecessary
details.

\subsubsection{Transformation of the jump graph to 
the steepest-descent directions}

According the steepest-descent strategy, we first 
transform the jump contour for each of the RH problems 
to the steepest-descent graph for the exponential 
$\exp\{\int^{\xi}\mu(z)\,dz\}$,
see Figure~\ref{fig2}.
%%%%%%%%%%%%%%%%%%%%%%%%%%
\begin{figure}[htb]
\begin{center}
\mbox{\epsfig{figure=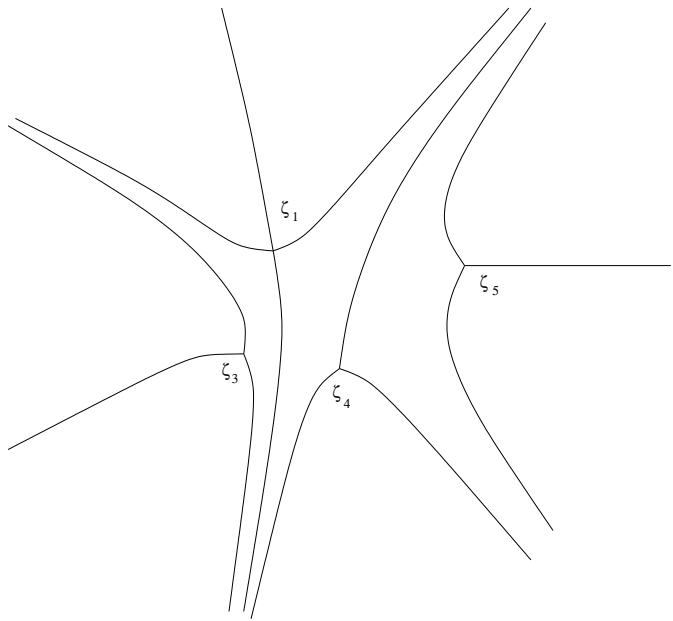,width=0.7\textwidth}}
\end{center}
\caption{Typical anti-Stokes lines for the special solution 
of \PItwo\ as $x_0\in D_0$, $\Im x_0>0$.}
\label{fig2}
\end{figure}
%%%%%%%%%%%%%%%%%%%%%%%%%%

Observe that, in the special case (\ref{tri-truncated_Stokes}),
$s_{\pm2}=s_{\pm1}=0$, $s_0=s_{\pm3}=-i$,
the jump graph depicted on Figure~\ref{fig1}
can be transformed to the one shown in
Figure~\ref{fig3},
%%%%%%%%%%%%%%%%%%%%%%%%%%
\begin{figure}[htb]
\begin{center}
\mbox{\epsfig{figure=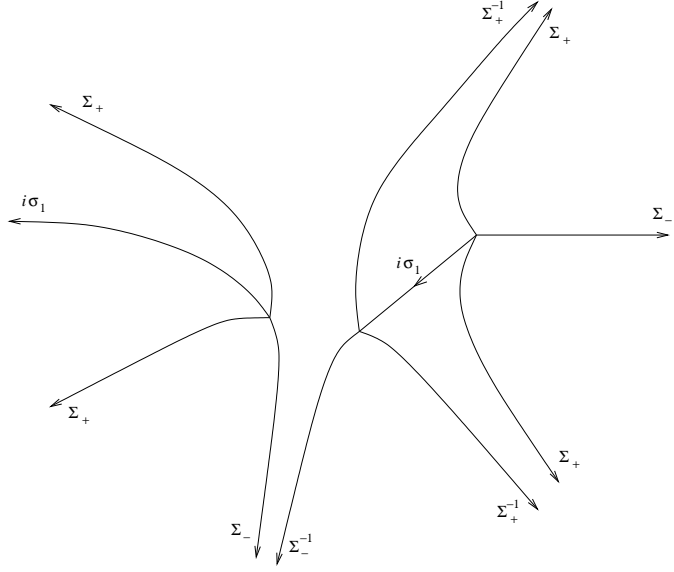,width=0.7\textwidth}}
\end{center}
\caption{The jump graph for the RH problem in the special case
$s_{\pm2}=s_{\pm1}=0$, $s_0=s_{\pm3}=-i$.}
\label{fig3}
\end{figure}
%%%%%%%%%%%%%%%%%%%%%%%%%%
where
\begin{equation}\label{jump_matrices_tt}
\Sigma_-:=
\begin{pmatrix}
1&0\\
-i&1
\end{pmatrix},
\quad
\Sigma_+:=
\begin{pmatrix}
1&-i\\
0&1
\end{pmatrix}.
\end{equation}

\subsubsection{Model elliptic curve and abelian integrals}

The large $t$ asymptotics of $\Psi(\lambda)$ corresponding 
to the Stokes multipliers (\ref{tri-truncated_Stokes}) 
is constructed on the Riemann surface $\Gamma$ of the 
model elliptic curve
\begin{equation}\label{surface_def}
w^2=(\xi-\xi_3)(\xi-\xi_4)(\xi-\xi_5),
\end{equation}
glued of two copies of the complex $\xi$-plane cut along 
$[\xi_5,\xi_4]\cup[\xi_3,-\infty)$,
see Figure~\ref{fig4}.
%%%%%%%%%%%%%%%%%%%%%%%%%%
\begin{figure}[htb]
\begin{center}
\mbox{\epsfig{figure=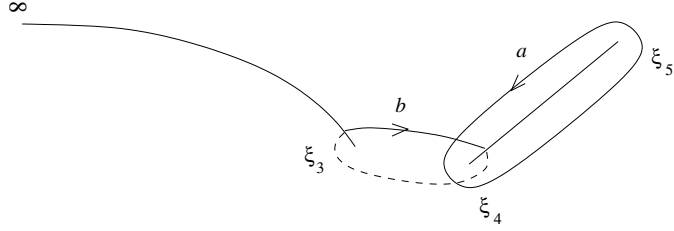,width=0.7\textwidth}}
\end{center}
\caption{The Riemann surface $\Gamma$ and the basis
of cycles $a,b$.}
\label{fig4}
\end{figure}
%%%%%%%%%%%%%%%%%%%%%%%%%%

Define the complete elliptic integrals
\begin{multline}\label{AB_def}
{\mathcal A},{\mathcal B}=
\oint_{a,b}\mu(\xi)\,d\xi
=\tfrac{1}{30}\oint_{a,b}
(\xi-\xi_1)w(\xi)\,d\xi,
% \\
% \shoveleft{
% {\mathcal B}=\oint_b\mu(\xi)\,d\xi
% =\tfrac{1}{30}\oint_b(\xi-\xi_1)w(\xi)\,d\xi,
% }\hfill
\\
\shoveleft{
\omega_{a,b}=\oint_{a,b}\frac{d\xi}{w(\xi)},\quad
\tau=\frac{\omega_b}{\omega_a},\quad
\Im\tau>0,
}\hfill
\end{multline}
and abelian integrals
\begin{equation}\label{hat_g_U_def}
g(\xi)=
\int_{\xi_5}^{\xi}\mu(z)\,dz,\quad
U(\xi)=\frac{1}{\omega_a}\int_{\xi_5}^{\xi}
\frac{dz}{w(z)}.
\end{equation}
We define the integral $g(\xi)$ on the upper sheet of 
the Riemann surface $\Gamma$ cut along the sum of
intervals $[\xi_5,\xi_4]\cup[\xi_4,\xi_3]\cup[\xi_3,-\infty)$.

The Boutroux conditions (\ref{Boutroux_tt}) imply
\begin{equation}\label{Boutroux}
{\mathcal A},{\mathcal B}\in i{\mathbb R}.
\end{equation}

Observe the following properties of $g(\xi)$
and $U(\xi)$:
\begin{enumerate}
\item
as $\xi\to\infty$,
\begin{multline}\label{hat_g_infty_as}
U(\xi)=U_{\infty}+{\mathcal O}(\xi^{-1/2}),\quad
U_{\infty}=-\tfrac{1}{2}\tau,
\\
\shoveleft{
g(\xi)=\vartheta+g_{\infty}
+{\mathcal O}(\xi^{-1/2}),\quad
\vartheta=
\tfrac{1}{105}\xi^{7/2}
-\tfrac{1}{3}\xi^{3/2}
+x_0\xi^{1/2},\quad
g_{\infty}=
-\tfrac{1}{2}{\mathcal B},
}\hfill
\end{multline}
\item
$g(\xi)$ and $U(\xi)$ are discontinuous across the broken
line $[\xi_5,\xi_4]\cup[\xi_4,\xi_3]\cup[\xi_3,-\infty)$
oriented from $\xi_5$ to infinity, moreover
\begin{multline}\label{hat_g_jump}
\xi\in(\xi_5,\xi_4)\colon\quad
g_+(\xi)+g_-(\xi)=0,\quad
U_+(\xi)+U_-(\xi)=0,
\\
\shoveleft{
\xi\in(\xi_4,\xi_3)\colon\quad
g_+(\xi)-g_-(\xi)=-{\mathcal A},\quad
U_+(\xi)-U_-(\xi)=-1,
}\hfill
\\
\shoveleft{
\xi\in(\xi_3,\infty)\colon\quad
g_+(\xi)+g_-(\xi)=-{\mathcal B},\quad
U_+(\xi)+U_-(\xi)=-\tau.
}\hfill
\\
\end{multline}
\end{enumerate}

\subsection{``External'' parametrix}

In this subsection, following \cite{FIKN}, we solve 
the permutation RH problem on the segments
$[\xi_5,\xi_4]\cup[\xi_3,-\infty)$ 
whose solution gives a leading order contribution to 
the solution of the above RH problem.

\begin{RHP}\label{quasi_permutation_RHP}
Given $t\gg1$ and the complex value of the parameter $x_0$,
find a piece-wise holomorphic $2\times2$ 
matrix function $\Phi^{(\nu)}(\xi)$ with the following properties:
\begin{equation}\label{PhiBA_as_def}
\hskip-0.in
(1)\qquad
\lim_{\xi\to\infty}
\tfrac{1}{\sqrt2}(\sigma_3+\sigma_1)
\xi^{-\frac{\nu}{4}\sigma_3}t^{-\frac{\nu}{8}\sigma_3}
\Phi^{(\nu)}(\xi)
e^{-t^{7/4}\vartheta\sigma_3}=I,
\end{equation}
where
\begin{equation}\label{vartheta_new_def}
\vartheta=\tfrac{1}{105}\xi^{7/2}
-\tfrac{1}{3}\xi^{3/2}
+x_0\xi^{1/2};
\end{equation}
across the union of segments $(\xi_5,\xi_4)\cup(\xi_3,-\infty)$ 
oriented as indicated, the jump condition holds true,
\begin{equation}\label{PsiBA_jump_rel}
\Phi^{(\nu)}_+(\xi)=\Phi^{(\nu)}_-(\xi)i\sigma_1,\quad
\xi\in(\xi_5,\xi_4)\cup(\xi_3,-\infty)
\end{equation}
where $\Phi^{(\nu)}_+(\xi)$ and $\Phi^{(\nu)}_-(\xi)$ are limits of 
$\Phi^{(\nu)}(\xi)$ on the segments from the left and from the right,
respectively, see Figure~\ref{fig3}.
\end{RHP}

We do not impose any conditions on the behavior of $\Phi^{(\nu)}(\xi)$
at the points $\xi=\xi_j$, $j=3,4,5$. As the result, the solution of 
the model RH problem~\ref{quasi_permutation_RHP} is 
determined up to a left rational matrix multiplier with possible 
poles at $\xi=\xi_j$, $j=3,4,5$, and certain asymptotics at infinity. 
Below, we will use this rational multiplier to prove or disprove the asymptotic 
solvability of the original RH problem~\ref{nu_RHP}.

\subsubsection{$\beta$-factor}
Consider the principal branches of the functions $\beta_{\nu}(\xi)$,
\begin{align}\label{beta1_def}
&\beta_{-1}(\xi)=(\xi-\xi_3)^{-1/4}(\xi-\xi_4)^{1/4}(\xi-\xi_5)^{-1/4}
\\
\label{beta2_def}
&\beta_{3}(\xi)=(\xi-\xi_3)^{1/4}(\xi-\xi_4)^{1/4}(\xi-\xi_5)^{1/4}
\end{align}
both defined on the $\xi$ complex plane with cuts along the broken line
$[\xi_5,\xi_4]\cup[\xi_4\xi_3]\cup[\xi_3,-\infty)$. These functions
solve the scalar RH problems:
\begin{enumerate}
\item
$\beta_{\nu}(\xi)=\xi^{\frac{\nu}{4}}(1+{\mathcal O}(\xi^{-1}))$,
as $\xi\to\infty$;
\item
the discontinuity of $\beta_{\nu}(\xi)$ across the oriented contour
$[\xi_5,\xi_4]\cup[\xi_4\xi_3]\cup[\xi_3,-\infty)$ 
is described by the conditions
\begin{multline}\label{beta1_jumps}
\xi\in(\xi_5,\xi_4)\colon\quad
\beta_{-1}^+(\xi)=i\beta_{-1}^-(\xi),\quad
\beta_3^+(\xi)=-i\beta_3^-(\xi),
\\
\shoveleft{
\xi\in(\xi_4,\xi_3)\colon\quad
\beta_{-1}^+(\xi)=\beta_{-1}^-(\xi),\quad
\beta_3^+(\xi)=-\beta_3^-(\xi),
}\hfill
\\
\shoveleft{
\xi\in(\xi_3,-\infty)\colon\quad
\beta_{-1}^+(\xi)=i\beta_{-1}^-(\xi),\quad
\beta_3^+(\xi)=i\beta_3^-(\xi).
}\hfill
\end{multline}
\end{enumerate}

\subsubsection{The Riemann theta function and 
the Baker--Akhiezer functions}

Define one more function,
\begin{equation}\label{h_nu_def}
h_{\nu}(\xi)=
t^{7/4}(g(\xi)-g_{\infty})
+\delta_{\nu}(U(\xi)-U_{\infty}),\quad
g_{\infty}=-\tfrac{1}{2}{\mathcal B},\quad
U_{\infty}=-\tfrac{1}{2}\tau,
\end{equation}
where the parameter $\delta_{\nu}$ is defined by
\begin{equation}\label{delta_nu_def}
\delta_{-1}=-t^{7/4}{\mathcal A},\quad
\delta_{3}=-t^{7/4}{\mathcal A}+i\pi.
\end{equation}
This function $h_{\nu}(\xi)$ has the following obvious properties,

(1) as $\xi\to+\infty$,
\begin{equation}\label{h_nu_infty_as}
h_{\nu}(\xi)=
t^{7/4}\vartheta
+{\mathcal O}(\xi^{-1/2});
\end{equation}

(2) $h_{\nu}(\xi)$ is discontinuous across the broken
line $[\xi_5,\xi_4]\cup[\xi_4,\xi_3]\cup[\xi_3,-\infty)$
oriented from $\xi_5$ to infinity, moreover
\begin{multline}\label{h_nu_jump}
\xi\in(\xi_5,\xi_4)\colon\quad
h_{\nu}^+(\xi)+h_{\nu}^-(\xi)
=t^{7/4}{\mathcal B}
+\delta_{\nu}\tau,
\\
\shoveleft{
\xi\in(\xi_4,\xi_3)\colon\quad
h_{\nu}^+(\xi)-h_{\nu}^-(\xi)=
-t^{7/4}{\mathcal A}-\delta_{\nu},
}\hfill
\\
\shoveleft{
\xi\in(\xi_3,\infty)\colon\quad
h_{\nu}^+(\xi)+h_{\nu}^-(\xi)=0.
}\hfill
\end{multline}

Using the Riemann theta-function,
 $\Theta(z)=\sum_ne^{\pi in^2\tau+2\pi i nz}$,
define the matrix function $\Phi_{\nu}^{(BA)}(\xi)$,
\begin{multline}\label{Phi_BA_def}
\Phi_{\nu}^{(BA)}(\xi)=
(\beta_{\nu}(\xi))^{\sigma_3}
\tfrac{1}{\sqrt{2}}
\times
\\
\times
\begin{pmatrix}
\tfrac{\Theta(U(\xi)+V+\phi_{\nu})}
{\Theta(U(\xi)+\frac{1+\tau}{2})}
c_1(\xi)&
\tfrac{\Theta(-U(\xi)+V+\phi_{\nu})}
{\Theta(-U(\xi)+\frac{1+\tau}{2})}
c_1^*(\xi)
\\
\tfrac{\Theta(U(\xi)+V+\phi_{\nu}-\frac{1+\tau}{2})}
{\Theta(U(\xi))}
c_2(\xi)&
\tfrac{\Theta(-U(\xi)+V+\phi_{\nu}-\frac{1+\tau}{2})}
{\Theta(-U(\xi))}
c_2^*(\xi)
\end{pmatrix}
e^{h_{\nu}(\xi)\sigma_3}.
\end{multline}
Here the parameters $V$, $\phi_{\nu}$ and the factors 
$c_j(\xi)$, $c_j^*(\xi)$ are defined by
\begin{multline}\label{V_phi_nu_cj_cj*_def}
V=-\tfrac{1}{2\pi i}t^{7/4}(\tau{\mathcal A}-{\mathcal B});
\\
\shoveleft
\phi_{-1}=\tfrac{1+\tau}{2},\quad
\phi_3=0;
\\
\shoveleft
c_1(\xi)=
\tfrac{\Theta(\frac{1}{2})}{\Theta(V+\phi_{\nu}-\frac{\tau}{2})},\quad
c_1^*(\xi)=\tfrac{\Theta(\frac{1}{2}+\tau)}
{\Theta(V+\phi_{\nu}+\frac{\tau}{2})}
\quad\mbox{if}\quad
V+\phi_{\nu}\neq\tfrac{1}{2}+n+m\tau,
\\
\shoveleft
c_1(\xi)=\beta_{-1}^{-2}(\xi)
\tfrac{\omega_a\Theta(\frac{1}{2})}{2\Theta'(\frac{1+\tau}{2})},\quad
c_1^*(\xi)=\beta_{-1}^{-2}(\xi)
\tfrac{\omega_a\Theta(\frac{1}{2}+\tau)}
{2\Theta'(\tfrac{1+\tau}{2})}
\quad\mbox{if}\quad
V+\phi_{\nu}=\tfrac{1}{2},\hfill
\\
c_2(\xi)=\tfrac{\Theta(\frac{\tau}{2})}
{\Theta(V+\phi_{\nu}-\frac{1}{2}-\tau)},\quad
c_2^*(\xi)=\tfrac{\Theta(\frac{\tau}{2})}
{\Theta(V+\phi_{\nu}-\frac{1}{2})}
\quad\mbox{if}\quad
V+\phi_{\nu}\neq\tfrac{\tau}{2}+n+m\tau,
\\
\shoveleft
c_2(\xi)=\beta_{-1}^{-2}(\xi)
\tfrac{\omega_a\Theta(\frac{\tau}{2})}
{2\Theta'(\frac{1+\tau}{2})},\quad
c_2^*(\xi)=\beta_{-1}^{-2}(\xi)
\tfrac{\omega_a\Theta(\frac{\tau}{2})}
{2\Theta'(\frac{1+\tau}{2})}
\quad\mbox{if}\quad
V+\phi_{\nu}=\tfrac{\tau}{2},\hfill
\\
n,m\in{\mathbb Z}.
\end{multline}
It can be shown that $\det\Phi_{\nu}^{(BA)}(\xi)\equiv-1$,
the function $\Phi_{\nu}^{(BA)}(\xi)$ (\ref{Phi_BA_def})
satisfies (\ref{PhiBA_as_def}) and (\ref{PsiBA_jump_rel})
and thus is one of the solutions of the 
RH problem~\ref{quasi_permutation_RHP}. 
Any other solution to this RH problem has the form of the product
\begin{equation}\label{PsiBA_gen_sol}
\Psi_{\nu}^{(BA)}(\xi)=R_{\nu}(\xi)\Phi_{\nu}^{(BA)}(\xi),
\end{equation}
where $R_{\nu}(\xi)$ is rational with poles at $\xi=\xi_j$, $j=3,4,5$,
and satisfies the asymptotic condition 
$R_{\nu}(\xi)=I+{\mathcal O}(\xi^{-1})$
as $\xi\to\infty$.

\subsection{Local RH problem solution near the branch points
$\xi=\xi_j$, $j=5,4,3$}

As it is well known, near a single branch points, the relevant
boundary-value problem can be solved using the classical Airy 
functions, see e.g.\ \cite{FIKN}.

\subsubsection{RH problem for the Airy functions}

Define the Wronsky matrix of the Airy functions
\cite{BE},
\begin{equation}\label{Airy_matrix}
Z_0(z)=\sqrt{2\pi}e^{-i\pi/4}
\begin{pmatrix}
v_2(z)&v_1(z)\\
\frac{d}{dz}v_2(z)&\frac{d}{dz}v_1(z)
\end{pmatrix}
e^{-i\pi\sigma_3/4},
\end{equation}
where
\begin{equation}\label{v12_def}
v_1(z)=\Ai(z),\quad
v_2(z)=e^{i2\pi/3}\Ai(e^{i2\pi/3}z).
\end{equation}
Besides $Z_0(z)$, introduce auxiliary functions
\begin{multline}\label{Airyj_def}
Z_{-1}(z)=Z_0(z)(\Sigma_+)^{-1},\quad
Z_1(z)=Z_0(z)\Sigma_-,\quad
Z_2(z)=Z_1(z)\Sigma_+,
\\
\Sigma_+=
\begin{pmatrix}
1&-i\\
0&1
\end{pmatrix},\quad
\Sigma_-=
\begin{pmatrix}
1&0\\
-i&1
\end{pmatrix}.
\end{multline}
By construction \cite{BE},
\begin{multline}\label{Airy_matrix_as}
Z_j(z)=
z^{-\sigma_3/4}
\tfrac{1}{\sqrt2}(\sigma_3+\sigma_1)
(I+{\mathcal O}(z^{-\frac{3}{2}}))
e^{\frac{2}{3}z^{3/2}\sigma_3},
\\
z\to\infty,\quad
z\in\omega_j=\bigl\{
z\in{\mathbb C}\colon\
\arg z\in\bigl(-\pi+\tfrac{2\pi}{3}j,\tfrac{\pi}{3}+\tfrac{2\pi}{3}j)
\bigr\}.
\end{multline}

Assemble the piece-wise holomorphic functions $Z^{(j)}(z)$,
$j=5,4,3$,
\begin{equation}\label{Zj_RH_def}
Z^{(j)}(z)=
\begin{cases}
Z_{-1}(z)G_j,\quad
\arg z\in(-\pi,-\tfrac{2\pi}{3}),
\\
Z_{0}(z)G_j,\quad
\arg z\in(-\tfrac{2\pi}{3},0),
\\
Z_{1}(z)G_j,\quad
\arg z\in(0,\tfrac{2\pi}{3}),
\\
Z_{2}(z)G_j,\quad
\arg z\in(\tfrac{2\pi}{3},\pi),
\end{cases}
\end{equation}
where
\begin{equation*}
G_5=G_3=
I,\quad
G_4=
i\sigma_3.
\end{equation*}

Observing that the jump matrices of $Z^{(j)}(z)$ coincide with those
across the lines emanating from the node points
$\xi_j$, $j=5,4,3$, in Figure~\ref{fig3}, we are ready 
to construct the relevant local parametrices,
\begin{equation}\label{parametrix_j_def}
\Psi_{\nu}^{(j)}(\xi)=B^{(\nu)}_j(\xi)Z^{(j)}(z^{(j)}(\xi)),\quad
j=5,4,3,\quad
\nu=-1,3.
\end{equation}
Here $B^{(\nu)}_j(\xi)$, $j=5,4,3$, are holomorphic in some finite
neighborhoods of $\xi=\xi_j$ matrices, and $z=z^{(j)}(\xi)$ are
changes of variables biholomorphic in some finite neighborhoods of 
$\xi=\xi_j$.

\subsubsection{Determination of the local change $z=z^{(j)}(\xi)$}

This biholomorphic change of variables has to be chosen in a way 
to ensure that the global parametrix, see below, has small enough 
jumps as $t\to+\infty$, namely it must satisfy the condition
\begin{equation*}
\tfrac{2}{3}(z^{(j)}(\xi))^{3/2}=
t^{7/4}(g(\xi)-g(\xi_j))
+\delta_{\nu}(U(\xi)-U(\xi_j))
+o(1),\quad
|\xi-\xi_j|=const.
\end{equation*}

The biholomorphicity condition is satisfied with the choice 
\begin{equation}\label{z_j_from_xi_change}
z^{(j)}(\xi)=
t^{7/6}
\Bigl(
\tfrac{3}{2}
\int_{\hat\xi_j}^{\xi}
\hat\mu(z)\,dz
\Bigr)^{2/3},\quad
j=5,4,3,
\end{equation}
where $\hat\mu(\xi)$ has the form (\ref{spectral_curve_p12}),
(\ref{spectral_curve_genus1}) with the branch points $\hat\xi_j$,
% 
% Besides the explained above asymptotic objects,
% consider the finite $t$ model functions and integrals,
% \begin{multline}\label{t-finite_objects}
% \hat w^2(\xi)=(\xi-\hat\xi_3)(\xi-\hat\xi_4)(\xi-\hat\xi_5),\quad
% \hat\mu(\xi)=\tfrac{1}{30}(\xi-\hat\xi_1)\hat w(\xi),
% \\
% \hat{\mathcal A},\hat{\mathcal B}=
% \oint_{a,b}\hat\mu(z)\,dz,\quad
% \hat\omega_{a,b}=\oint_{a,b}\frac{dz}{\hat w(z)},\quad
% \hat\tau=\frac{\hat\omega_b}{\hat\omega_a},
% \\
% \hat g(\xi)=\int_{\hat\xi_5}^{\xi}\hat\mu(z)\,dz,\quad
% \hat U(\xi)=\frac{1}{\hat\omega_a}
% \int_{\hat\xi_5}^{\xi}\frac{dz}{\hat w(z)},
% \\
% \hat\xi_j=\xi_j+{\mathcal O}(t^{-7/4}),\quad
% j=1,3,4,5.
% \end{multline}
% We treat (\ref{t-finite_objects}) as 1-parameter families of deformations 
% of the original functions and integrals described as follows. Consider 
% the asymptotic spectral curve (\ref{spectral_curve_p12}) satisfying 
% the degeneration condition,
\begin{multline*}
\hat\mu^2=
\tfrac{1}{900}\xi^5
-\tfrac{1}{30}\xi^3
+\tfrac{1}{30}x_0\xi^2
+\tfrac{1}{30}\hat D_1\xi
+\tfrac{1}{30}\hat D_0=
\\
=\tfrac{1}{900}
(\xi-\hat\xi_1)^2
(\xi-\hat\xi_3)(\xi-\hat\xi_4)(\xi-\hat\xi_5).
\end{multline*}
% 
% Two independent holomorphic differentials 
% on the genus 2 hyperelliptic curve are obtained by
% differentiation of the meromorphic differential
% $\hat\mu\,d\xi$ with respect to the parameters
% $\hat D_1$ and $\hat D_0$,
% \begin{equation*}
% d\hat U_j=\tfrac{\partial}{\partial\hat D_j}\hat\mu(\xi)\,d\xi=
% \frac{\xi^j\,d\xi}{2(\xi-\hat\xi_1)\hat w(\xi)},\quad
% j=0,1.
% \end{equation*}
The unique abelian differential holomorphic on the elliptic curve
$\hat w^2=(\xi-\hat\xi_3)(\xi-\hat\xi_4)(\xi-\hat\xi_5)$ is 
\begin{equation*}
% \tfrac{2}{\hat\omega_a}\,d\hat U_1
% -\tfrac{2}{\hat\omega_a}\hat\xi_1\,d\hat U_0=
d\hat U=\tfrac{2}{\hat\omega_a}
\Bigl(\tfrac{\partial}{\partial\hat D_1}
-\hat\xi_1\,\tfrac{\partial}{\partial\hat D_0}
\Bigr)\hat\mu(\xi)\,d\xi.
\end{equation*}
Thus the 1-parameter deformation of the model degenerated curve 
that respects the degeneration is generated by the vector field
\begin{equation}\label{1-parameter_vector_field}
\tfrac{\partial}{\partial D}=
\tfrac{2}{\hat\omega_a}
\Bigl(\tfrac{\partial}{\partial\hat D_1}
-\hat\xi_1\,\tfrac{\partial}{\partial\hat D_0}
\Bigr).
\end{equation}
Finally, we find the elliptic curve satisfying the asymptotic 
conditions,
\begin{multline}\label{AB1tau_to_hatAhatB}
\mu(\xi)+t^{-7/4}\frac{\delta_{\nu}}{w(\xi)}=
\hat\mu(\xi)+{\mathcal O}(t^{-7/2}(\xi-\xi_j)^{-3/2}),
\\
\mbox{as}\quad
|\xi-\xi_1|>C_1t^{-7/8},\quad
|\xi-\xi_j|>C_jt^{-7/4},\quad
C_j=const,\quad
j=3,4,5,
\\
\shoveleft{
\omega_{a,b}=\hat\omega_{a,b}+{\mathcal O}(t^{-7/2}),\quad
\hat\tau=\tau+{\mathcal O}(t^{-7/2}),
}\hfill
\\
{\mathcal A}+t^{-7/4}\delta_{\nu}=
\hat{\mathcal A}+{\mathcal O}(t^{-7/2}),\quad
{\mathcal B}+t^{-7/4}\tau\delta_{\nu}=
\hat{\mathcal B}+{\mathcal O}(t^{-7/2}).
\end{multline}

\subsection{Global parametrix}

The global approximate solution to the RH problem for $\Phi^{(\nu)}(\xi)$,
$\nu=-1,3$,
is a piece-wise analytic matrix function $\tilde\Psi_{\nu}(\xi)$ defined as follows,
\begin{multline}\label{global_param_def}
\tilde\Psi_{\nu}(\xi)=
\begin{cases}
\Psi_{\nu}^{(j)}(\xi),\quad
\xi\in C_j,\\
\Psi_{\nu}^{(BA)},\quad
\xi\in{\mathbb C}\backslash\cup_jC_j,
\end{cases}
\\
\Psi_{\nu}^{(BA)}(\xi)=R_{\nu}(\xi)\Phi^{(BA)}(\xi),\quad
C_j=\{\xi\in{\mathbb C}\colon|\xi-\xi_j|<r\},
\\
0<r<\tfrac{1}{2}\min_{k\neq j}|\xi_k-\xi_j|,\quad
k,j=3,4,5,
\end{multline}
see Figure~\ref{fig13}.
%%%%%%%%%%%%%%%%%%%%%%%%%%%%%%%%%%%%%%
\begin{figure}[htb]
\begin{center}
\mbox{\epsfig{figure=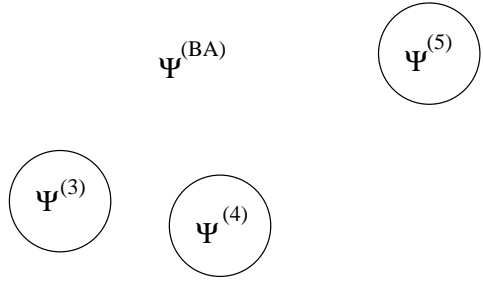,width=0.5\textwidth}}
\end{center}
\caption{The construction of the global parametrix.}
\label{fig13}
\end{figure}
%%%%%%%%%%%%%%%%%%%%%%%%%%%%%%%%%%%%%%

The exact solution is constructed using the correction
function $\chi(\xi)$,
\begin{equation}\label{chi_for_Psi_def}
\chi(\xi)
=\Psi(\xi)
\tilde\Psi^{-1}(\xi),
\end{equation}
that satisfies the following RH problem:

(1) the limit
\begin{equation*}
\lim_{\xi\to\infty}
\xi^{1/2}\bigl(
\tfrac{1}{\sqrt2}(\sigma_3+\sigma_1)
\xi^{\frac{1}{4}\sigma_3}
t^{\frac{1}{8}\sigma_3}
\chi(\xi)
t^{-\frac{1}{8}\sigma_3}
\xi^{-\frac{1}{4}\sigma_3}
\tfrac{1}{\sqrt2}(\sigma_3+\sigma_1)
-I
\bigr)
\end{equation*}
exists and is diagonal;

(2) across the contour $\gamma$ shown in Figure~\ref{fig14},
the jump condition holds, $\chi_+(\xi)=\chi_-(\xi)H(\xi)$
where
\begin{equation*}
H(\xi)=\tilde\Psi_-(\xi)S(\xi)\tilde\Psi_+^{-1}(\xi).
\end{equation*}
%%%%%%%%%%%%%%%%%%%%%%%%%%%%%%%%%%%%%%
\begin{figure}[htb]
\begin{center}
\mbox{\epsfig{figure=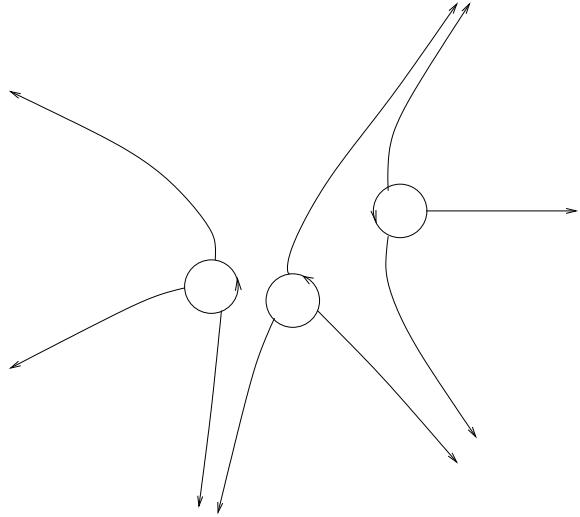,width=0.6\textwidth}}
\end{center}
\caption{The jump contour $\gamma$ for the correction function $\chi(\xi)$.}
\label{fig14}
\end{figure}
%%%%%%%%%%%%%%%%%%%%%%%%%%%%%%%%%%%%%%

To apply the $L^2$-theory to the latter RH problem \cite{Zh}, all jump matrices 
$H(\xi)$ across the jump contour $\gamma$ in Figure~\ref{fig14} must satisfy 
the estimate $\|H(\xi)-I\|_{L^2(\gamma)\cap C(\gamma)}=o(1)$ as $t\to\infty$. 
For the infinite tails emanating from the circles, this fact holds true because 
the relevant jumps are uniformly exponentially small. 

The jumps across the circles centered at the branch points 
$\xi_j$, $j=3,4,5$, can be made small as $t$ is large if one adjusts 
the rational matrix $R_{\nu}(\xi)$, $\nu=-1,3$, and the holomorphic matrices 
$B_j^{(\nu)}(\xi)$, $\nu=-1,3$, $j=5,4,3$, in an appropriate way. Omitting 
the straightforward but tedious computations, we present the result:
\begin{multline}\label{Rm1_sol}
V-\tfrac{1}{2},
V-\tfrac{\tau}{2},
V-\tfrac{1+\tau}{2}\neq
n+m\tau,\quad
n,m\in{\mathbb Z}\colon
\\
R_{-1}(\xi)=
t^{-\frac{1}{8}\sigma_3}
\begin{pmatrix}
1&\frac{1}{\xi-\xi_5}q^{(\infty)}
\\
r^{(\infty)}&
1+\frac{1}{\xi-\xi_5}q^{(\infty)}r^{(\infty)}
\end{pmatrix},
\\
q^{(\infty)}=
-\tfrac{\omega_a}{2}
(\xi_5-\xi_4)
\tfrac{e^{2\pi iV(t)}
\Theta(V+\frac{1+\tau}{2})
\Theta(V+\frac{\tau}{2})
\Theta(\frac{1}{2})\Theta(0)}
{\Theta(V+\frac{1}{2})\Theta(V)
\Theta(\frac{\tau}{2})\Theta'(\frac{1+\tau}{2})},
\\
r^{(\infty)}-q^{(\infty)}=
\tfrac{2\Theta(V)\Theta(V+\frac{1+\tau}{2})
\Theta'(\frac{1+\tau}{2})\Theta(0)}
{\omega_a\Theta(V+\frac{1}{2})\Theta(V+\frac{\tau}{2})
\Theta(\frac{1}{2})\Theta(\frac{\tau}{2})},
\\
\shoveleft
V=\tfrac{1}{2}\colon\quad
R_{-1}(\xi)=t^{-\frac{1}{8}\sigma_3}
\Bigl(\frac{\xi-\xi_5}{\xi-\xi_4}\Bigr)^{\sigma_3},
\\
\shoveleft{
V=\tfrac{\tau}{2}\colon\quad
R_{-1}(\xi)=t^{-\frac{1}{8}\sigma_3}.
}\hfill
\end{multline}
The relevant correction function satisfies the estimates
\begin{equation}\label{chi_estim}
\chi(\xi)-I=
\begin{cases}
{\mathcal O}(t^{-7/24}),\quad
|\xi|<\mbox{const},
\\
{\mathcal O}(t^{-7/24}\xi^{-1}),\quad
|\xi|\to\infty,
\end{cases}
\end{equation}
in the domain of the parameter $x_0=xt^{-3/2}$ described by 
\begin{equation}\label{V_bad}
|V-\tfrac{1+\tau}{2}-n-m\tau|>Ct^{-\frac{7}{24}+\epsilon},\quad
n,m\in{\mathbb Z},\quad
\epsilon=\mbox{const}>0.
\end{equation}

\subsection{Large $t$ asymptotics of the RH 
problem~\ref{nu_RHP}($\nu=3$)}

In contrast, the computation of the left rational multiplier
$R_{\nu}(\xi)$ for $\nu=3$ yields 
\begin{equation}\label{R3_sol}
R_3(\xi)=
t^{\frac{3}{8}\sigma_3}I,
\end{equation}
together with the additional condition,
\begin{equation}\label{R3_exists_constraint}
\Theta(V)=0,
\end{equation}
or, equivalently,
\begin{equation}\label{V_for_poles}
V=\tfrac{1+\tau}{2}+n+m\tau.
\end{equation}
In this case, the ``external'' parametrix is elementary,
\begin{equation}\label{Phi_BA2_1+tau/2_def}
\Phi_3^{(BA)}(\xi)=
(\beta_3(\xi))^{\sigma_3}
\tfrac{1}{\sqrt{2}}
(\sigma_3+\sigma_1)
e^{h_3(\xi)\sigma_3},
\end{equation}
and the correction function $\chi(\xi)$
satisfies the estimates (\ref{chi_estim}).

\begin{rem}
The above computation shows that, as $t\to+\infty$, 
for each value of the deformation parameter $x$, $xt^{-3/2}\in D_0$,
the RH problem~\ref{nu_RHP} is solvable either with $\nu=-1$ or $\nu=3$.
Thus there is no room for the solvability of the RH problem with 
$\nu=-5$. Therefore, at least in the large $t$ limit, the special solution 
to \PItwo\ has no triple pole collisions and the corresponding solutions to 
\PItwone\ have no branch points. The relevant Malgrange divisor 
consists of smooth branches only.
\end{rem}

%\begin{conj}
%The solution $a(t)$ of \PItwone\ corresponding to the special solution
%of \PItwo\ has no branch point for $\forall\, t\in{\mathbb C}$.
%\end{conj}

\section{Pole asymptotic distribution}

To compute the large $t$ asymptotic distribution of poles
of the special solution to \PItwo,
we use the phase shift given in (\ref{V_for_poles}) and 
the definition of $V$ in (\ref{V_phi_nu_cj_cj*_def}),
\begin{equation}\label{pole_distribution_as}
\tfrac{1}{2\pi i}t^{7/4}(\tau{\mathcal A}-{\mathcal B})=
n+\tfrac{1}{2}
+(m+\tfrac{1}{2})\tau.
\end{equation}
Using the canonical dissection of the Riemann surface,
the difference $\omega_b{\mathcal A}-\omega_a{\mathcal B}$ 
is expressed in terms of a single contour integral over the contour 
${\mathcal L}$ depicted in Figure~\ref{fig15},
%%%%%%%%%%%%%%%%%%%%%%%%%%%%%%%%%%%%%%
\begin{figure}[htb]
\begin{center}
\mbox{\epsfig{figure=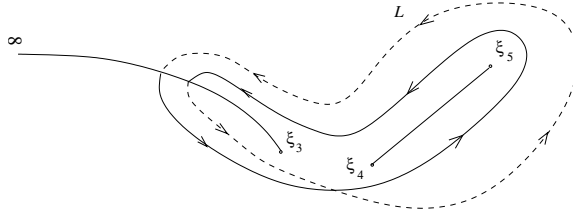,width=0.6\textwidth}}
\end{center}
\caption{The contour for the canonical dissection of the elliptic curve $\Gamma$.}
\label{fig15}
\end{figure}
%%%%%%%%%%%%%%%%%%%%%%%%%%%%%%%%%%%%%%
\begin{equation*}
\omega_b{\mathcal A}-\omega_a{\mathcal B}=
\omega_a\oint_{\mathcal L}U(z)\mu(z)\,dz.
\end{equation*}
Inflating the contour ${\mathcal L}$, we transform it to a contour
encircling the infinite point which is a branch point of the curve.
Then expanding the integrand at the infinity
and using the residue theorem, we find
\begin{equation}\label{tau_A-B_sol}
\tau{\mathcal A}-{\mathcal B}=
\oint_{\mathcal L}(U(z)-U_{\infty})\mu(z)\,dz
=-\tfrac{8\pi i}{7\omega_a}
(3x_0+\tfrac{2}{3}\xi_1),
\end{equation}
and the pole distribution formula (\ref{pole_distribution_as})
yields
\begin{equation}\label{pole_distribution_as_fin}
3x_0+\tfrac{2}{3}\xi_1=
\tfrac{7}{4}t^{-7/4}
\bigl(
(n+\tfrac{1}{2})\omega_a
+(m+\tfrac{1}{2})\omega_b
\bigr).
\end{equation}
Since $\xi_1$ is determined by $x_0$ via the Boutroux equations
(\ref{Boutroux_tt}), equation (\ref{pole_distribution_as_fin}) 
determines the pole position,
$x_0^{(n,m)}=x_{n,m}t^{-3/2}$ as a transcendent function
of two integers $(n,m)$. As $t\to+\infty$, the particular pole 
$x_0^{(n,m)}$ approaches the attracting point satisfying
the equation
\begin{equation}\label{attractor}
\xi_1(x_0)=-\tfrac{9}{2}x_0.
\end{equation}

\subsection{Quasi-stationary solutions of equation \PItwone}

Although it is not known how to solve the transcendent equation 
(\ref{attractor}), the problem of finding the attractor to the pole 
distribution significantly simplifies observing that the pole
attractors correspond to the quasi-stationary solutions 
of equation \PItwone,
\begin{equation}\label{a_quasi_stat_as}
a(t)=a_{\infty}t^{3/2}(1+{\mathcal O}(t^{-\epsilon})),\quad
\epsilon>0.
\end{equation}
Substituting (\ref{a_quasi_stat_as}) into \PItwone\ 
(\ref{a_ODE_from_KDV}), we find an algebraic equation
for the parameter $a_{\infty}$,
\begin{equation}\label{a_infty_eq}
a_{\infty}^4-\tfrac{236}{243} a_{\infty}^2+\tfrac{160}{2187}=0,
\end{equation}
with the roots
$a_{\infty}\in\{\pm\tfrac{2\sqrt5}{9\sqrt3},\pm\tfrac{2\sqrt2}{3}\}$.
However only one of the roots,
\begin{equation}\label{a_infty_sol}
a_{\infty}=\tfrac{2\sqrt5}{9\sqrt3},
\end{equation}
is consistent with the described above properties of the large $t$
asymptotic spectral curve
for the special solution of \PItwo. The linearization of equation
\PItwone at the 0-parameter power series solution with the leading order 
term coefficient (\ref{a_infty_sol}) has four linearly independent solutions. 
Two of them are exponential, 
$\sim\exp[\pm i\tfrac{2\sqrt2}{3}(\tfrac{5}{3})^{3/4}\tfrac{4}{7}t^{7/4}]$,
and we set them aside. Two other solutions of the linearized equation, 
$\sim t^{-1/4}$ and $\sim t^{-1/4}\ln t$, are relevant to our quasi-stationary
behavior of the poles. Using them, we form the 2-parameter series
\begin{multline}\label{quasi_stat_series}
a(t)=t^{3/2}
\sum_{k=0}^{\infty}
(t^{-\frac{7}{4}}\ln(t^{-\frac{7}{4}}))^k
a_k(t),\quad
a_k(t)=\sum_{l=0}^{\infty}a_{kl}t^{-\frac{7}{4}l},
\\
\shoveleft
a_{00}=\tfrac{2\sqrt5}{9\sqrt3},
\\
\shoveleft
a_{10},a_{01}\in{\mathbb C}\quad
\mbox{are arbitrary},
\\
\shoveleft
a_{20}=-\tfrac{3\sqrt{15}}{784} a_{10}^2,\quad
a_{11}=-\tfrac{3\sqrt{15}}{392}(a_{01}-2a_{10})a_{10},
\\
\hfill
a_{02}=\tfrac{3\sqrt{3}
(\sqrt{15}-50a_{01}^2+200 a_{01}a_{10}+3130a_{10}^2)}
{7840\sqrt{5}},
\\
\shoveleft
a_{30}=\tfrac{1467}{307328}a_{10}^3,\quad
a_{21}=\tfrac{27}{307328}(163 a_{01}+861a_{10})a_{10}^2,
\\
a_{12}=\tfrac{27}{9834496}(1907\tfrac{\sqrt{3}}{\sqrt{5}}+5216 a_{01}^2
+55104 a_{01} a_{10}+233392 a_{10}^2) a_{10},
\\
\hfill
a_{03}=\tfrac{9}{3073280}
(1630 a_{01}^3
+25830 a_{01}^2a_{10}
+218805 a_{01} a_{10}^2
+788362 a_{10}^3)
\\
\hfill
+\tfrac{9\sqrt{3}}{49172480\sqrt{5}}
(28605 a_{01}+258143 a_{10}),\quad
\\
\shoveleft{
\dots
}\hfill
\end{multline}
Let us relate the free parameters $a_{01},a_{10}$ in 
(\ref{quasi_stat_series}) with the integers $n$ and $m$ in 
(\ref{pole_distribution_as_fin}).
Recall that, along the boundary $\partial D_+$, two branch points 
$\xi_4$ and $\xi_5$ of the asymptotic spectral curve coalesce. 
Namely, if $t\to+\infty$, the limiting 
values corresponding to the attracting point are
\begin{equation}\label{attractor_branch_points}
x_0^*=\tfrac{2\sqrt5}{9\sqrt3},\quad
\xi_3^*=-\tfrac{4\sqrt5}{\sqrt3},\quad
\xi_{1,2}^*=-\tfrac{\sqrt5}{\sqrt3},\quad
\xi_{4,5}^*=\sqrt{15}.
\end{equation}
The asymptotics of the branch points in the model elliptic spectral 
curve compatible with the expansion (\ref{quasi_stat_series}) is given by
\begin{multline}\label{model_xi12345}
\xi_1=\xi_2=-\tfrac{\sqrt5}{\sqrt3}
+\tfrac{3}{4}
(a_{01}
+7a_{10}
+a_{10}\ln( t^{-7/4})
)t^{-7/4}
+{\mathcal O}(t^{-7/2}\ln^2t),
\\
\shoveleft
\xi_3=-\tfrac{4\sqrt5}{\sqrt3}
-\tfrac{6}{7}
\bigl(
a_{01}
+4a_{10}
+a_{10}\ln( t^{-7/4})
\bigr)t^{-7/4}
+{\mathcal O}(t^{-7/2}\ln^2t),\quad
\\
\shoveleft
\xi_{4,5}=\sqrt{15}
\mp\sqrt6\sqrt[4]{15}\sqrt{a_{10}}\,t^{-7/8}
\\
-\tfrac{9}{28}
\bigl(
a_{01}
+11a_{10}
+a_{10}\ln( t^{-7/4})
\bigr)t^{-7/4}
+{\mathcal O}(t^{-21/8}\ln^2t).
\end{multline}
The asymptotics of the periods $\omega_{a,b}$ as
$\xi_5-\xi_4\to0$ have the forms
\begin{multline}\label{omega_ab_deg_as}
\omega_a=
\tfrac{2\pi i}{\sqrt{\xi_4-\xi_3}}
(1-\tfrac{\xi_5-\xi_4}{4(\xi_4-\xi_3)}+{\mathcal O}((\xi_5-\xi_4)^2)),
\\
\shoveleft
\omega_b=
\tfrac{2}{\sqrt{\xi_4-\xi_3}}
\ln\tfrac{\xi_5-\xi_4}{16(\xi_4-\xi_3)}
+{\mathcal O}((\xi_5-\xi_4)\ln(\xi_5-\xi_4)),\quad
\xi_5-\xi_4\to0.
\end{multline}
Thus, for large $t$,
\begin{multline}\label{omega_ab_att_as}
\omega_a=i\pi\tfrac{2\sqrt[4]{3}}{\sqrt{7}\sqrt[4]{5}}
(1+{\mathcal O}(t^{-7/8})),
\\
\shoveleft
\omega_b=
-\tfrac{\sqrt[4]{3}\sqrt{7}}{4\sqrt[4]{5}}
\ln t
+\tfrac{2\sqrt[4]3}{\sqrt{7}\sqrt[4]{5}}
\ln\tfrac{3^{3/2}\sqrt{-a_{10}}}
{2^{3/2}7^{3/2}\sqrt[4]{15}}
+{\mathcal O}(t^{-7/8}\ln t),\quad
t\to+\infty.
\end{multline}

Finally, the coefficients $a_{10}$ and $a_{01}$
determining the asymptotic series for $a(t)$ 
(\ref{quasi_stat_series}),
\begin{equation*}
a(t)=
\tfrac{2\sqrt5}{9\sqrt3}t^{3/2}
+t^{-1/4}
\bigl(
a_{01}
+a_{10}\ln( t^{-7/4})
\bigr)
+\dots,
\end{equation*}
follow using the asymptotic formula (\ref{pole_distribution_as_fin}),
\begin{multline}\label{quasi_stat_from_pole}
a_{10}\ln( t^{-7/4})+a_{01}+\dots
=
\bigl(
-\tfrac{2\sqrt5}{9\sqrt3}
-\tfrac{2}{9}\xi_1
\bigr)t^{7/4}
+\tfrac{7}{12}
(
(n+\tfrac{1}{2})\omega_a
+(m+\tfrac{1}{2})\omega_b
)
=
\\
=-\tfrac{1}{6}a_{10}\ln( t^{-7/4})
-\tfrac{1}{6}a_{01}-\tfrac{7}{6}a_{10}
\\
+\tfrac{7}{12}
(n+\tfrac{1}{2})i\pi\tfrac{2\sqrt[4]{3}}{\sqrt{7}\sqrt[4]{5}}
+\tfrac{7}{12}
(m+\tfrac{1}{2})
\tfrac{\sqrt[4]{3}}{\sqrt[4]{5}\sqrt{7}}
\bigl[
\ln( t^{-7/4})
+\ln\tfrac{3^{3}a_{10}}
{2^{5}7^2\sqrt{15}}
\bigr]
+\dots
\end{multline}
Equating coefficients of $\ln( t^{-7/4})$, we find 
$a_{10}=a_{10}^{(m,n)}$,
\begin{equation}\label{a10_sol}
a_{10}^{(m,n)}=
(m+\tfrac{1}{2})
\tfrac{\sqrt[4]{3}}{2\sqrt[4]{5}\sqrt{7}},\quad
m\in{\mathbb Z}_+,
\end{equation}
while the constant terms yield the family of values of
$a_{01}=a_{01}^{(m,n)}$,
\begin{multline}\label{a01_sol}
a_{01}^{(m,n)}=
(m+\tfrac{1}{2})
\tfrac{\sqrt[4]{3}}{2\sqrt[4]{5}\sqrt{7}}
\ln\bigl[
(m+\tfrac{1}{2})
\tfrac{3^{11/4}}{2^{6}5^{3/4}7^{5/2}e}
\bigr]
+i\pi(n+\tfrac{1}{2})\tfrac{\sqrt[4]{3}}{\sqrt[4]{5}\sqrt{7}},
\\
m\in{\mathbb Z}_+,\quad
n\in{\mathbb Z}.
\end{multline}
The formulas (\ref{a10_sol}) and (\ref{a01_sol})
with (\ref{quasi_stat_series}) yield the asymptotic formula
(\ref{pole_dynamics_sol}) for $a^{(m,n)}(t)$ which implies that 
the poles of the special solution to \PItwo\ in a vicinity of
the attracting point $x_0^*=\tfrac{2\sqrt5}{9\sqrt3}$
form a regular lattice with the slowly modulated intervals
and a boundary formed by the line of poles corresponding
to the values $m=0$ and $n\in{\mathbb Z}$. In particular,
the interval between two the most right vertical lines of poles
is given by
\begin{equation*}
a^{(0,n)}(t)-a^{(1,n)}(t)=
-t^{-1/4}
\tfrac{\sqrt[4]{3}}{2\sqrt[4]{5}\sqrt{7}}
\ln\bigl[
 t^{-7/4}\tfrac{3^{17/4}}{2^{7}5^{3/4}7^{5/2}e}
\bigr]
+{\mathcal O}(t^{-2}\ln^2t),
\end{equation*}

Note that the boundary $\partial D_+$ formally
corresponds to $m=-\tfrac{1}{2}$, $n\in{\mathbb R}$,
and the distance between the first vertical line of poles
and $\partial D_+$ is 
\begin{equation*}
a^*(t)
-a^{(0,n)}(t)=
-t^{-1/4}
\tfrac{\sqrt[4]{3}}{4\sqrt[4]{5}\sqrt{7}}
\ln\bigl[
 t^{-7/4}
\tfrac{3^{11/4}}{2^{7}5^{3/4}7^{5/2}e}
\bigr]
+{\mathcal O}(t^{-2}\ln^2t).
\end{equation*}

\section{Problems and perspectives}

Above, we have presented and studied equation \PItwone, 
to our best knowledge, the first of the differential equations that controls
the isomonodromy deformations of a linear ODE with rational
coefficients and does not possess the Painlev\'e property. 
For the first glance, its existence contradicts the theorem by Miwa and
Malgrange. However, it is not the case. The absence of the Painlev\'e 
property in \PItwone\ is related to the fact that the domain of the solvability
of the corresponding RH problem in the 2-dimensional complex space with
the coordinates $(t,x)$ is restricted to the Malgrange divisor of \PItwo, 
i.e.\ to the set of the complex lines $(t,a(t))$ which are allowed to intersect. 
Actually, this fact provides us with the important information on the nontrivial 
analytic structure of the Malgrange divisor for \PItwo\ which forms a Riemann 
surface with the infinite number of sheets and third order branch points.

Besides this, the discovery of equation \PItwone\ provides
us with a new wide field of research. For instance, it is interesting 
to explore the possibility of 
existence of similar equations associated with other isomonodromic 
solutions of KdV or other integrable PDEs like Nonlinear Schr\"odinger 
or Pohlmeyer--Lund--Regge equations.

Other interesting problem not discussed above is the structure in 
\PItwone\ induced by the singularity reduction of the Hamiltonian 
structure in \PItwo\ (and, in the case of a successful extension 
of the singularity reduction methodology to the hierarchies 
associated with other Painlev\'e equations, the structures induced 
by the Weyl symmetries).

Finally, we mention the problem of characterization of such initial
data to \PItwo\ (and other isomonodromy deformation equations)
whose singularity reductions do not have branch points.
In the \PItwo\ case, we conjecture that its special solution
considered above does not have merging poles for any $t$
and therefore the relevant solution of \PItwone\ does not have
branch points. If this conjecture is true, it can serve as one
more characteristic property of this special solution.

\bigskip
{\bf Acknowledgments.} The authors are grateful to Prof. Y.~Ohyama for the reference to the paper \cite{shimomura}. This work is
partially supported by the European Research Council Advanced Grant FroM-PDE, by the Russian Federation Government Grant No. 2010-220-01-077 and by PRIN 2010-11 Grant ``Geometric and analytic theory of Hamiltonian systems in finite and infinite dimensions'' of Italian Ministry of Universities and Researches.
A.K. thanks the staff of SISSA for hospitality
during his visit when this work was done.

%%%%%%%%%%%%%%%%%%%%%%%%%
\bibliographystyle{plain}
%\bibliography{}
\ifx\undefined\bysame
\newcommand{\bysame}{\leavevmode\hbox to3em{\hrulefill}\,}
\fi

\end{document}